\documentclass[10pt,a4paper]{amsart}

\usepackage{amssymb}
\usepackage{amsmath}
\usepackage[hidelinks]{hyperref}
\usepackage{enumitem}

\newtheorem{thm}{Theorem}[section]
\newtheorem{prop}[thm]{Proposition}
\newtheorem{lem}[thm]{Lemma}
\newtheorem{cor}[thm]{Corollary}

\theoremstyle{definition}
\newtheorem{definition}[thm]{Definition}

\theoremstyle{remark}
\newtheorem{remark}[thm]{Remark}

\numberwithin{equation}{section}

\newcommand{\Q}{\mathbb{Q}} 
\newcommand{\Z}{\mathbb{Z}} 
\newcommand{\C}{\mathbb{C}} 
\newcommand{\h}{\mathbb{H}} 
\newcommand{\SL}{\mathrm{SL}_2(\Z)} 
\newcommand{\Gal}{\mathrm{Gal}} 

\DeclareMathOperator{\re}{Re} 

\begin{document}
	
	\title[Equations in three singular moduli]{Equations in three singular moduli: the equal exponent case}
	\author{Guy Fowler}
	\address{Leibniz Universit\"{a}t Hannover, Institut für Algebra, Zahlentheorie und Diskrete Mathematik, Welfengarten 1, 30167 Hannover, Germany.}
	\email{\href{mailto:fowler@math.uni-hannover.de}{fowler@math.uni-hannover.de}}
	\urladdr{\url{https://www.guyfowler.uk/}}
	\date{\today}
	\thanks{\textit{Acknowledgements}: The author would like to thank Jonathan Pila for helpful comments and the referee for a careful reading of this paper. The author has received funding from the European Research Council (ERC) under the European Union’s Horizon 2020 research and innovation programme (grant agreement No. 945714).}
	
	\begin{abstract}
	Let $a \in \Z_{>0}$ and $\epsilon_1, \epsilon_2, \epsilon_3 \in \{\pm 1\}$. We classify explicitly all singular moduli $x_1, x_2, x_3$ satisfying either $\epsilon_1 x_1^a + \epsilon_2 x_2^a + \epsilon_3 x_3^a \in \Q$ or $(x_1^{\epsilon_1} x_2^{\epsilon_2} x_3^{\epsilon_3})^{a} \in \Q^{\times}$. In particular, we show that all the solutions in singular moduli $x_1, x_2, x_3$ to the Fermat equations $x_1^a + x_2^a + x_3^a= 0$ and $x_1^a + x_2^a - x_3^a= 0$ satisfy $x_1 x_2 x_3 = 0$. Our proofs use a generalisation of a result of Faye and Riffaut on the fields generated by sums and products of two singular moduli, which we also establish.
	\end{abstract}
	
	\maketitle

\section{Introduction}\label{sec:intro}

Denote by $\h$ the complex upper half plane. Let $j \colon \h \to \C$ be the modular $j$-function. A singular modulus is a complex number $j(\tau)$, where $\tau \in \h$ is such that $[\Q(\tau) : \Q]=2$. Equivalently, a singular modulus is the $j$-invariant of an elliptic curve (over $\C$) with complex multiplication. Singular moduli are algebraic integers which generate the ring class fields of imaginary quadratic fields. For background, see e.g. \cite[Chapter~3]{Cox89}.

Let $n \in \Z_{>0}$. A CM-point of $\C^n$ is a point $(\sigma_1, \ldots, \sigma_n) \in \C^n$ such that each $\sigma_i$ is a singular modulus. A special subvariety of $\C^n$ is an irreducible component of some subvariety of $\C^n$ which is defined by equations of the form $x_i = \sigma$ and $\Phi_N(x_l, x_k)=0$, where $\sigma$ is a singular modulus and $N \in \Z_{>0}$. Here $\Phi_N$ denotes the $N$th classical modular polynomial (see \cite[\S11]{Cox89}). In particular, a CM-point is a zero-dimensional special subvariety.

Let $V \subset \C^n$ be a subvariety. The Andr\'e--Oort conjecture for $\C^n$, which was proved by Pila \cite{Pila11} (see Pila's paper and the references therein for background on the conjecture), states that $V$ contains only finitely many maximal special subvarieties. In particular, $V$ contains only finitely many CM-points which do not lie on positive-dimensional special subvarieties of $V$. Pila's theorem is ineffective; it does not provide an effective means of finding all the special subvarieties of a given $V$. For $n=2$, an effective version of the Andr\'e--Oort conjecture was proved by K\"uhne \cite{Kuhne12, Kuhne13} and Bilu, Masser, and Zannier \cite{BiluMasserZannier13} (see also Andr\'e's earlier ineffective proof \cite{Andre98} of the $n=2$ case). For $n > 2$, effective forms of the Andr\'e--Oort conjecture are known only for some restricted classes of subvarieties \cite{BiluKuhne20, Binyamini19}.

Alongside this work, there have also been a number of results which classify explicitly all the CM-points lying on particular families of algebraic curves contained in $\C^2$, e.g. \cite{AllombertBiluMadariaga15, BiluLucaMadariaga16, BiluLucaMadariaga20, LucaRiffaut19, Riffaut19}. In this article, we prove the following two results, which consider instead two families of algebraic surfaces in $\C^3$. 

\begin{thm}\label{thm:sum}
	Let $a \in \Z_{>0}$ and $\epsilon_1, \epsilon_2, \epsilon_3 \in \{\pm 1\}$. Let $x_1, x_2, x_3$ be singular moduli. Then 
	\[ \epsilon_1 x_1^a + \epsilon_2 x_2^a + \epsilon_3 x_3^a \in \Q \]
	if and only if one of the following holds:
	\begin{enumerate}
		\item $x_1, x_2, x_3 \in \Q$;
		\item some $x_i \in \Q$ and, for the remaining $x_j, x_k$, one has that $x_j = x_k$ and $\epsilon_j + \epsilon_k =0$;
		\item some $x_i \in \Q$ and the remaining $x_j, x_k$ are distinct, of degree $2$, and conjugate over $\Q$ with $\epsilon_j = \epsilon_k$;
		\item $\epsilon_1 = \epsilon_2 = \epsilon_3$ and $x_1, x_2, x_3$ are pairwise distinct, of degree $3$, and conjugate over $\Q$.
	\end{enumerate}
\end{thm}

\begin{thm}\label{thm:prod}
	Let $x_1, x_2, x_3$ be singular moduli and $a \in \Z \setminus \{0\}$. Then 
	\[(x_1 x_2 x_3)^a \in \Q^\times\]
	if and only if one of the following holds:
	\begin{enumerate}
		\item $x_1, x_2, x_3 \in \Q^\times$;
		\item some $x_i \in \Q^\times$ and the remaining $x_j, x_k$ are distinct, of degree $2$, and conjugate over $\Q$;
		\item $x_1, x_2, x_3$ are pairwise distinct, of degree $3$, and conjugate over $\Q$.
	\end{enumerate}
\end{thm}

Since $0=j(e^{2\pi i /3})$ is a singular modulus, we exclude the case of product $0$ in Theorem~\ref{thm:prod}. There are $13$ rational singular moduli (including $0$), $58$ singular moduli of degree $2$, and $75$ singular moduli of degree $3$. The list of these may be computed straightforwardly in PARI \cite{PARI2}. Theorems~\ref{thm:sum} and \ref{thm:prod} are thus completely explicit. The $a=1$ case of Theorem~\ref{thm:prod} was proved previously in the author's paper \cite{Fowler20}. To our knowledge, this is the only other explicit result about CM-points lying on some algebraic surfaces in $\C^3$ to have appeared previously.

We prove the following corollary of Theorem~\ref{thm:sum}, which gives a ``Fermat's last theorem'' for singular moduli.

\begin{cor}\label{cor:sum}
	Let $x_1, x_2, x_3$ be singular moduli. Suppose that there exists $a \in \Z_{>0}$ such that either $x_1^a + x_2^a + x_3^a =0$ or $x_1^a + x_2^a = x_3^a$. Then $x_1 x_2 x_3 = 0$.
\end{cor}

We also obtain the following generalisation of Theorem~\ref{thm:prod}. To state it, we need a definition. Weinberger \cite[Theorem~1]{Weinberger73} proved that there is at most one quadratic imaginary field $K$ such that the class group of $K$ is isomorphic to $(\Z / 2 \Z)^n$ for some $n \in \Z_{>0}$ and the fundamental discriminant $D_K$ of $K$ satisfies $\lvert D_K \rvert > 5460$.

\begin{definition}\label{def:tat}
	Let $K_*$ be the single exceptional quadratic imaginary field which has class group isomorphic to $(\Z / 2 \Z)^n$ for some $n \in \Z_{>0}$ and fundamental discriminant $D_*$ such that $\lvert D_* \rvert > 5460$, if such a field exists.
\end{definition}

Our result is then as follows.

\begin{thm}\label{thm:quot}
	Let $a \in \Z \setminus \{0\}$ and $\epsilon_1, \epsilon_2, \epsilon_3 \in \{\pm 1\}$. Let $x_1, x_2, x_3$ be singular moduli. If
	\[ (x_1^{\epsilon_1} x_2^{\epsilon_2} x_3^{\epsilon_3})^a \in \Q^\times, \]
	then one of the following holds:
	\begin{enumerate}
		\item $x_1, x_2, x_3 \in \Q^\times$;
		\item some $x_i \in \Q^\times$ and, for the remaining $x_j, x_k$, one has that $x_j = x_k$ and $\epsilon_j + \epsilon_k =0$;
		\item some $x_i \in \Q^\times$ and the remaining $x_j, x_k$ are distinct, of degree $2$, and conjugate over $\Q$ with $\epsilon_j = \epsilon_k$;
		\item $\epsilon_1 = \epsilon_2 = \epsilon_3$ and $x_1, x_2, x_3$ are pairwise distinct, of degree $3$, and conjugate over $\Q$;
		\item all of the following hold:
		\begin{enumerate}
			\item some $x_i, x_j$ are distinct and conjugate over $\Q$;
			\item $x_i = j(\nu)$ for some $\nu \in \h$ such that $\Q(\nu) \neq K_*$;
			\item the remaining $x_k$ is equal to $j(\tau)$ for some $\tau \in \h$ such that $\Q(\tau) = K_*$;
			\item $\epsilon_i=\epsilon_j=-\epsilon_k$;
			\item $\Q(x_k) \subset \Q(x_i)= \Q(x_j)$;
			\item $[\Q(x_i) : \Q(x_k)]=2$ and $[\Q(x_k) : \Q] \geq 128$.
		\end{enumerate}
	\end{enumerate}
\end{thm}

It is possible that the field $K_*$ does not exist. For example, the Generalised Riemann Hypothesis implies that there is no such field $K_*$. In this case, Theorem~\ref{thm:quot} would hold with case (5) removed, and would evidently then be an ``if and only if''.

The plan of this article is as follows. In Section~\ref{sec:background}, we recall those results about singular moduli which we will need. Section~\ref{sec:prim} contains some primitive element theorems for singular moduli, which will be central to our proofs. These generalise results of Faye and Riffaut \cite{FayeRiffaut18}. Theorem~\ref{thm:sum} is proved in Sections~\ref{sec:sum} and \ref{sec:diff}. Corollary~\ref{cor:sum} is proved in Section~\ref{sec:fermat}. Theorem~\ref{thm:prod} is then proved in Section~\ref{sec:prod}, and finally the proof of Theorem~\ref{thm:quot} is contained in Section~\ref{sec:quot}. Our proofs make extensive use of computations in PARI \cite{PARI2}; scripts are available from \url{https://github.com/guyfowler/three_singular_moduli}.

While this article was under review, Bilu, Gun, and Tron \cite[Theorem~1.1]{BiluGunTron22} proved a result related to Theorem~\ref{thm:quot}. They established explicit bounds on the triples $(x_1, x_2, x_3)$ of distinct singular moduli such that $x_1^a x_2^b x_3^c \in \Q^\times$ for some $a,b,c \in \Z \setminus \{0\}$. They list all the known examples of such triples $(x_1, x_2, x_3)$, but currently their bounds do not suffice to show that these are the only examples. 

In the particular case of Theorem~\ref{thm:quot}, a result of Bilu, Gun, and Tron \cite[Corollary~2.12]{BiluGunTron22} implies that case (5) of Theorem~\ref{thm:quot} cannot occur. Thus, Theorem~\ref{thm:quot} holds with case (5) removed and is thus evidently an ``if and only if''. One therefore obtains a complete list of those triples $(x_1, x_2, x_3)$ of singular moduli such that $x_1^a x_2^b x_3^c \in \Q^\times$ for some $a, b, c \in \Z\setminus \{0\}$ with $\lvert a \rvert = \lvert b \rvert = \lvert c \rvert$. 

\section{Background}\label{sec:background}

\subsection{Basic properties of singular moduli}\label{subsec:singmod}

For background on singular moduli, the reader is referred to e.g. \cite[Chapter~3]{Cox89}. We will use the properties of singular moduli which are stated in this section throughout this article, often without special reference.

Let $x = j(\tau)$ be a singular modulus, so that $\tau \in \h$ is such that $[\Q(\tau) : \Q] = 2$. The discriminant $\Delta$ of $x$ is defined by $\Delta = b^2 - 4 a c$, where $(a,b,c) \in \Z^3 \setminus \{(0,0,0)\}$ are such that $a \tau^2 + b \tau + c = 0$ and $\gcd(a,b,c)=1$. Observe that $\Delta<0$ and $\Delta \equiv 0, 1 \bmod 4$. We may write $\Delta = f^2 D$ for some unique $f \in \Z_{>0}$, where $D$ is the (fundamental) discriminant of the imaginary quadratic field $\Q(\sqrt{\Delta}) = \Q(\tau)$.

Let $x_1, \dots, x_n$ be the list of all the singular moduli of discriminant $\Delta$. The Hilbert class polynomial $H_\Delta(z)$ is defined to be
\[H_\Delta(z) = (z- x_1) \cdots (z-x_n).\] 
Remarkably, $H_\Delta(z)$ is a monic polynomial with coefficients in $\Z$ which is irreducible over $\Q$ and over $\Q(\sqrt{\Delta})$. In particular, every singular modulus is an algebraic integer. Further, the singular moduli of a given discriminant $\Delta$ form a complete set of Galois conjugates both over $\Q$ and over $\Q(\sqrt{\Delta})$.

The number of singular moduli of a given discriminant $\Delta$ is equal to $h(\Delta)$, the class number of the (unique) imaginary quadratic order of discriminant $\Delta$. In particular, for $x$ a singular modulus of discriminant $\Delta$, one thus has that
\[ h(\Delta) = [\Q(x) : \Q] = [ \Q(\sqrt{\Delta}, x) : \Q(\sqrt{\Delta})].\]
The field $\Q(\sqrt{\Delta}, x)$ is the ring class field (see \cite[\S9]{Cox89}) of the imaginary quadratic order of discriminant $\Delta$. The extension $\Q(\sqrt{\Delta}, x) / \Q$ is always a Galois extension \cite[\S3.2]{AllombertBiluMadariaga15}.

The singular moduli of a given discriminant $\Delta$ may be explicitly described \cite[Proposition~2.5]{BiluLucaMadariaga16}. Write $T_\Delta$ for the set of triples $(a,b,c) \in \Z^3$ such that: $\gcd(a,b,c)=1$, $\Delta = b^2-4ac$, and either $-a < b \leq a < c$ or $0 \leq b \leq a = c$. Then there is a bijection, given by $(a,b,c) \mapsto j((b + \sqrt{\Delta})/2a)$, between $T_\Delta$ and the singular moduli of discriminant $\Delta$. This description follows from the definition of the discriminant $\Delta$ and the fact that the $j$-function restricts to a bijection $F_j \to \C$, where $F_j$ is the usual fundamental domain for the action of $\SL$ on $\h$, i.e.
\[ F_j = \{ z \in \h : -\frac{1}{2} < \re z \leq \frac{1}{2}, \lvert z \rvert \geq 1, \mbox{ and } \lvert z \rvert > 1 \mbox{ if } \re z < 0 \}.\]
Observe that $(b + \sqrt{\Delta})/2a \in F_j$ for every $(a,b,c) \in T_\Delta$.

We require the following elementary lemma.

\begin{lem}\label{lem:dom}
	For a given discriminant $\Delta$, there exists:
	\begin{enumerate}
		\item a unique singular modulus corresponding to a triple with $a=1$;
		\item at most two singular moduli corresponding to triples with $a=2$, and:
		\begin{enumerate}
			\item if $\Delta \equiv 0, 4 \bmod 16$, then there are no such singular moduli;
			\item if $\Delta \equiv 1 \bmod 8$ and $\Delta \notin \{-7, -15\}$, then there are exactly two such singular moduli;
		\end{enumerate}
		\item at most two singular moduli corresponding to triples with $a=3$;
		\item at most two singular moduli corresponding to triples with $a=4$;
		\item at most two singular moduli corresponding to triples with $a=5$;
		\item at most four singular moduli corresponding to triples with $a=6$;
		\item at most two singular moduli corresponding to triples with $a=7$;
		\item at most four singular moduli corresponding to triples with $a=8$;
		\item at most three singular moduli corresponding to triples with $a=9$;
		\item at most four singular moduli corresponding to triples with $a=10$;
		\item at most two singular moduli corresponding to triples with $a=11$.
	\end{enumerate}
\end{lem}

\begin{proof}
 (1) and (2) are \cite[Proposition~2.6]{BiluLucaMadariaga16}\footnote{That $\Delta = -15$ must be excluded in 2(b) is not stated in \cite{BiluLucaMadariaga16}, but is necessary since $h(-15)=2$.}. (3)--(5) are proved in \cite[Lemma~2.1]{Fowler20}. (6)--(11) then follow by the same argument as in \cite{Fowler20}, which uses the fact that $b_1^2-b_2^2 \equiv 0 \bmod 4a$ for any two tuples $(a, b_1, c_1), (a, b_2, c_2) \in T_{\Delta}$.
	\end{proof}

The singular modulus corresponding to the unique tuple $(a,b,c) \in T_{\Delta}$ with $a=1$ is called the ``dominant'' singular modulus of discriminant $\Delta$. A singular modulus corresponding to a tuple $(a, b, c) \in T_\Delta$ with $a=2$ is called a ``subdominant'' singular modulus of discriminant $\Delta$. 

\subsection{Bounds on singular moduli}\label{subsec:bd}

The $j$-function has a Fourier expansion
\[ j(z) = q^{-1} + 744 + \sum_{n = 1}^\infty c_n q^n,\]
where $q= \exp(2 \pi i z)$ and the $c_i \in \Z_{>0}$. Observe that it follows from this Fourier expansion that the dominant singular modulus of a given discriminant is always real.

Suppose that $x$ is the singular modulus corresponding to some triple $(a, b, c) \in T_\Delta$, so that $x = j((b + \sqrt{\Delta})/2a)$ and $(b + \sqrt{\Delta})/2a \in F_j$. It follows \cite[Lemma~1]{BiluMasserZannier13} from the Fourier expansion for the $j$-function that
\begin{align}\label{eq:ineq1}
(e^{\pi \lvert \Delta \rvert^{1/2}/a} - 2079) \leq \lvert x \rvert \leq (e^{\pi \lvert \Delta \rvert^{1/2}/a} + 2079).
\end{align}
This inequality explains the choice of the terms ``dominant'' and ``subdominant'' for singular moduli corresponding to triples $(a,b,c) \in T_\Delta$ with $a=1$ and $a=2$ respectively. This terminology was first used in \cite{AllombertBiluMadariaga15}.

In particular, the dominant singular modulus of a given discriminant $\Delta$ is much larger in absolute value than all the other singular moduli of discriminant $\Delta$. 

\begin{prop}[{\cite[Lemma~3.5]{AllombertBiluMadariaga15}}]\label{prop:dom}
	Let $x$ be the dominant singular modulus of discriminant $\Delta$. Suppose that $x'$ is a singular modulus of discriminant $\Delta$ with $x' \neq x$. Then $\lvert x' \rvert \leq 0.1 \lvert x \rvert$.
\end{prop}

(In \cite{AllombertBiluMadariaga15}, it is required that $\lvert \Delta \rvert \geq 11$. Nonetheless, the result holds without this restriction, since $\lvert \Delta \rvert < 11$ implies that $x \in \Q$ and for such $x$ the result is trivial.)

\begin{prop}\label{prop:dominc}
	Let $x$ be the dominant singular modulus of discriminant $\Delta$ and $x' \neq x$ a singular modulus of discriminant $\Delta'$, where $\lvert \Delta' \rvert \leq \lvert \Delta \rvert$. Then $\lvert x' \rvert < \lvert x \rvert$.
\end{prop}

\begin{proof}
	If $\Delta = \Delta'$, then this follows from Proposition~\ref{prop:dom}. So suppose that $\lvert \Delta' \rvert < \lvert \Delta \rvert$. Hence, $\lvert \Delta' \rvert \leq \lvert \Delta \rvert - 1$. We thus have that
	\[ \lvert x \rvert \geq (e^{\pi \lvert \Delta \rvert^{1/2}} - 2079)\]
	and 
	\[ \lvert x' \rvert \leq (e^{\pi (\lvert \Delta \rvert - 1)^{1/2}} + 2079).\]
	In particular, one may verify that $\lvert x' \rvert < \lvert x \rvert$ if $\lvert \Delta \rvert \geq 9$. The result then follows by inspecting the list of singular moduli with discriminant $<9$ in absolute value.
	\end{proof}

We also make use of the following lower bound for non-zero singular moduli, due to Bilu, Luca, and Pizarro-Madariaga \cite{BiluLucaMadariaga16} (see also its generalisation to the differences of distinct singular moduli in \cite{BiluFayeZhu19}).

\begin{lem}[{\cite[\S3.1]{BiluLucaMadariaga16}}]\label{lem:lower}
	Let $x$ be a non-zero singular modulus of discriminant $\Delta$. Then
	\[ \lvert x \rvert \geq \min \{ 4.4 \times 10^{-5}, 3500 \lvert \Delta \rvert^{-3}\}. \]
\end{lem}

\begin{remark}
The proof of this inequality in \cite{BiluLucaMadariaga16} is a step in a larger proof, in which the authors have previously assumed that the singular moduli considered are of degree at least $3$. This degree assumption though is not used anywhere in the proof of the inequality in Lemma~\ref{lem:lower}; this inequality thus holds for all non-zero singular moduli.
\end{remark}

\subsection{Fields generated by singular moduli}\label{subsec:fields}

Singular moduli generate the ring class fields of imaginary quadratic fields. In particular, it is a strong condition for two distinct singular moduli to generate closely related fields. We now collect some explicit results along these lines, which we need for our proof. The first result was proved mostly in \cite{AllombertBiluMadariaga15}, as Corollary~4.2 and Proposition~4.3 there. For the ``further'' claim in (2), see \cite[\S3.2.2]{BiluLucaMadariaga16}.

\begin{lem}\label{lem:samefield}
	Let $x_1, x_2$ be singular moduli with discriminants $\Delta_1, \Delta_2$ respectively. Suppose that $\Q(x_1) = \Q(x_2)$, and denote this field $L$. Then $h(\Delta_1) = h(\Delta_2)$, and we have that:
	\begin{enumerate}
		\item If $\Q(\sqrt{\Delta_1}) \neq \Q(\sqrt{\Delta_2})$, then the possible fields $L$ are listed in \cite[Table~2]{AllombertBiluMadariaga15}. Further, the extension $L / \Q$ is Galois and the discriminant of any singular modulus $x$ with $\Q(x) = L$ is also listed in this table.
		\item If $\Q(\sqrt{\Delta_1}) = \Q(\sqrt{\Delta_2})$, then either: $L = \Q$ and $\Delta_1, \Delta_2 \in \{-3, -12, -27\}$; or: $\Delta_1 / \Delta_2 \in \{1, 4, 1/4\}$. Further, if $\Delta_1 = 4 \Delta_2$, then $\Delta_2 \equiv 1 \bmod 8$.
	\end{enumerate}
\end{lem}

The next two results are from \cite{Fowler20}.

\begin{lem}[{\cite[Lemma~2.3]{Fowler20}}]\label{lem:subfieldsamefund}
	Let $x_1, x_2$ be singular moduli such that $\Q(x_1) \supset \Q(x_2)$ with $[\Q(x_1) : \Q(x_2)]=2$. Suppose that $\Q(\sqrt{\Delta_1}) = \Q(\sqrt{\Delta_2})$. Then either $x_2 \in \Q$, or $\Delta_1 \in \{9 \Delta_2 / 4, 4\Delta_2, 9 \Delta_2, 16\Delta_2\}$.	
\end{lem}

\begin{lem}\label{lem:subfielddiffund}
	Let $x_1, x_2$ be singular moduli such that $\Q(x_1) \supset \Q(x_2)$ with $[\Q(x_1) : \Q(x_2)]=2$. Suppose that $\Q(\sqrt{\Delta_1}) \neq \Q(\sqrt{\Delta_2})$. Then there exists $i \in \{1,2\}$ such that: the extension $\Q(x_i) / \Q$ is Galois and either $\Delta_i$ is listed in \cite[Table~1]{AllombertBiluMadariaga15} or $h(\Delta_i) \geq 128$.
\end{lem}

\begin{proof}
	This is a slightly stronger statement than \cite[Lemma~2.4]{Fowler20}, but is in fact exactly what is proved there.
	\end{proof}

We also need a result on the exceptional field $K_*$, which was defined in Definition~\ref{def:tat}.

\begin{lem}\label{lem:tat}
	If $x$ is a singular modulus of discriminant $\Delta$ such that the extension $\Q(x) / \Q$ is Galois and $h(\Delta) \geq 128$, then $\Q(\sqrt{\Delta}) = K_*$. 
\end{lem}

\begin{proof}
	This is immediate from \cite[\S2.2, Remark~2.3, and Corollary~3.3]{AllombertBiluMadariaga15}.
	\end{proof}

\subsection{Explicit Andr\'e--Oort results in two dimensions}\label{subsec:2dim}

We recall here those explicit Andr\'e--Oort results in two dimensions which we will make use of in the sequel. These are due to Riffaut \cite{Riffaut19} and Luca and Riffaut \cite{LucaRiffaut19}.

\begin{thm}[{\cite[Theorem~1.3]{LucaRiffaut19}}]\label{thm:sum2}
	Let $x_1, x_2$ be distinct singular moduli and $m, n \in \Z_{>0}$. Suppose that $A x_1^m + B x_2^n = C$ for some $A, B, C \in \Q^\times$. Then $\Q(x_1)=\Q(x_2)$ and this number field has degree at most $2$ over $\Q$.
\end{thm}

\begin{thm}[{\cite[Theorem~1.7]{Riffaut19}}]\label{thm:prod2}
	Let $x_1, x_2$ be non-zero singular moduli. Suppose that $x_1^m x_2^n \in \Q$ for some $m, n \in \Z \setminus \{0\}$. Then one of the following holds:
	\begin{enumerate}
		\item $x_1, x_2 \in \Q$;
		\item $x_1 = x_2$ and $m+n =0$;
		\item $m=n$ and $x_1, x_2$ are distinct, of degree $2$, and conjugate over $\Q$.
	\end{enumerate}
\end{thm}

The $m=n=1$ cases of these two results were proved previously in \cite{AllombertBiluMadariaga15} and \cite{BiluLucaMadariaga16} respectively. For partial results on the $x_1 = x_2$ case of Theorem~\ref{thm:sum2}, see \cite{BiluLucaMadariaga20}.

\section{Primitive element theorems for singular moduli}\label{sec:prim}

In \cite{FayeRiffaut18}, Faye and Riffaut proved the following primitive element theorems for sums and products of singular moduli. See also the generalisation of Theorem~\ref{thm:fieldsum} obtained in \cite{BiluFayeZhu19}.

\begin{thm}[{\cite[Theorem~4.1]{FayeRiffaut18}}]\label{thm:fieldsum}
	Let $x_1, x_2$ be distinct singular moduli of discriminants $\Delta_1, \Delta_2$ respectively. Let $\epsilon \in \{\pm 1\}$. Then $\Q(x_1 + \epsilon x_2) = \Q(x_1, x_2)$ unless $\Delta_1 = \Delta_2$ and $\epsilon =1 $. If $\Delta_1 = \Delta_2$, then $[\Q(x_1, x_2) : \Q(x_1 + x_2)] \leq 2$.
\end{thm}

\begin{thm}[{\cite[Theorem~5.1]{FayeRiffaut18}}]\label{thm:fieldprod}
	Let $x_1, x_2$ be distinct non-zero singular moduli of discriminants $\Delta_1, \Delta_2$ respectively. Let $\epsilon \in \{\pm 1\}$. Then $\Q(x_1 x_2^{\epsilon}) = \Q(x_1, x_2)$ unless $\Delta_1 = \Delta_2$ and $\epsilon =1 $. If $\Delta_1 = \Delta_2$, then $[\Q(x_1, x_2) : \Q(x_1 x_2)] \leq 2$.
\end{thm}

We also have the following result of Riffaut \cite{Riffaut19}. (In \cite{Riffaut19}, it is required that the discriminant $\Delta$ of $x$ satisfies $\lvert \Delta \rvert \geq 11$. We may dispense with this condition because the result is trivial for $\lvert \Delta \rvert < 11$, since all singular moduli having such discriminant are rational.)

\begin{prop}[{\cite[Lemma~2.6]{Riffaut19}}]\label{prop:fieldpower}
	Let $x$ be a singular modulus and $a \in \Z \setminus \{0\}$. Then $\Q(x^a) = \Q(x)$.
\end{prop}

In this section, we prove the following two results, which can be seen as combining the results of \cite{FayeRiffaut18} with Proposition~\ref{prop:fieldpower}. The proofs of Theorems~\ref{thm:fieldsumpower} and \ref{thm:fieldprodpower} both follow the approach of Faye and Riffaut. 

\begin{thm}\label{thm:fieldsumpower}
	Let $x_1, x_2$ be distinct singular moduli of discriminants $\Delta_1, \Delta_2$ respectively. Let $\epsilon \in \{\pm 1\}$ and $a \in \Z_{>0}$. Then $\Q(x_1^a + \epsilon x_2^a) = \Q(x_1, x_2)$, unless $\Delta_1 = \Delta_2$ and $\epsilon =1 $. If $\Delta_1 = \Delta_2$, then $[\Q(x_1, x_2) : \Q(x_1^a + x_2^a)] \leq 2$.
\end{thm}

\begin{thm}\label{thm:fieldprodpower}
	Let $x_1, x_2$ be distinct non-zero singular moduli of discriminants $\Delta_1, \Delta_2$ respectively. Let $\epsilon \in \{\pm 1\}$ and $a \in \Z \setminus \{0\}$. Then $\Q(x_1^a x_2^{\epsilon a}) = \Q(x_1, x_2)$, unless $\Delta_1 = \Delta_2$ and $\epsilon =1 $. If $\Delta_1 = \Delta_2$, then $[\Q(x_1, x_2) : \Q(x_1^a x_2^a)] \leq 2$.
\end{thm}

The proof of the $a=1$ case of our Theorem~\ref{thm:prod} in \cite{Fowler20} used Theorem~\ref{thm:fieldprod}. In this paper, our proofs of Theorems~\ref{thm:sum} and \ref{thm:prod} will use Theorems~\ref{thm:fieldsumpower} and \ref{thm:fieldprodpower} in a corresponding way.

\subsection{Proof of Theorem~\ref{thm:fieldsumpower}}

Let $x_1, x_2$ be distinct singular moduli of discriminants $\Delta_1, \Delta_2$ respectively. Write $D_1, D_2$ for their corresponding fundamental discriminants. Let $\epsilon \in \{\pm 1\}$ and $a \in \Z_{>0}$. Observe that Proposition~\ref{prop:fieldpower} implies that
\[ \Q(x_1, x_2) = \Q(x_1^a, x_2^a).\]
Let $L$ be the Galois closure of $\Q(x_1, x_2) / \Q$. Define
\[G= \Gal(L / \Q(x_1^a + \epsilon x_2^a))\]
and
\[ H = \Gal(L/ \Q(x_1, x_2)).\]
So $H \leq G$, and $H=G$ if and only if $\Q(x_1, x_2) = \Q(x_1^a + \epsilon x_2^a)$. 

Let $\sigma \in G$. Then 
\[ x_1^a + \epsilon x_2^a = \sigma(x_1^a) + \epsilon \sigma(x_2^a).\]
So $\sigma(x_1^a) = x_1^a$ if and only if $\sigma(x_2^a) = x_2^a$. Since $\Q(x_1^a)=\Q(x_1)$, we have that $\sigma(x_1^a) = x_1^a$ if and only if $\sigma(x_1) = x_1$. Similarly, $\sigma(x_2^a) = x_2^a$ if and only if $\sigma(x_2) = x_2$. Consequently,
\begin{align*}
	 H &= \{ \sigma \in G : \sigma(x_1^a)=x_1^a\}= \{ \sigma \in G : \sigma(x_1)=x_1\}\\
	 &= \{ \sigma \in G : \sigma(x_2^a)=x_2^a\}= \{ \sigma \in G : \sigma(x_2)=x_2\}.
	 \end{align*}
 
 If either $x_1 \in \Q$ or $x_2 \in \Q$, then the desired result follows from Proposition~\ref{prop:fieldpower}. So we will assume that $x_1, x_2 \notin \Q$.
 
 \subsubsection{The case where $\Delta_1 = \Delta_2$.} 
 Write $\Delta = \Delta_1 = \Delta_2$. We may assume, after applying an automorphism of $\C$ if necessary, that $x_1$ is dominant, and so $x_2$ is not dominant. By \eqref{eq:ineq1}, one therefore has that
 \begin{align}\label{eq:ineq2}
 \lvert x_1^a + \epsilon x_2^a \rvert &\geq \lvert x_1 \rvert^a - \lvert x_2 \rvert^a \nonumber \\
 	&\geq (e^{\pi \lvert \Delta \rvert^{1/2}} - 2079)^a - (e^{\pi \lvert \Delta \rvert^{1/2}/2} + 2079)^a.
 	\end{align}
 
 First, assume that $\epsilon = 1$. Suppose then that $[\Q(x_1, x_2) : \Q(x_1^a + x_2^a)] > 2$. Then $[G : H] > 2$. So there exists $\sigma \in G$ such that $\sigma(x_1^a) \neq x_1^a$ and $\sigma(x_1^a) \neq x_2^a$. Since $x_1^a + x_2^a = \sigma(x_1^a) + \sigma(x_2^a)$, we thus have that $\sigma(x_2^a) \neq x_1^a$. In particular, $\sigma(x_1), \sigma(x_2) \neq x_1$, and so neither of $\sigma(x_1), \sigma(x_2)$ is dominant.
 
 Therefore,
 \begin{align}\label{eq:ineq3}
 	\lvert x_1^a + x_2^a \rvert &= \lvert \sigma(x_1)^a + \sigma(x_2)^a \rvert \nonumber\\
 	&\leq 2 (e^{\pi \lvert \Delta \rvert^{1/2}/2} + 2079)^a.
  \end{align}
The bounds \eqref{eq:ineq2} and \eqref{eq:ineq3} are incompatible once 
\[ (e^{\pi \lvert \Delta \rvert^{1/2}} - 2079)^a > 3 (e^{\pi \lvert \Delta \rvert^{1/2}/2} + 2079)^a.\]
In particular, they are incompatible when
\[ \frac{e^{\pi \lvert \Delta \rvert^{1/2}} - 2079}{e^{\pi \lvert \Delta \rvert^{1/2}/2} + 2079} > 3,\]
which happens for all $\lvert \Delta \rvert \geq 9$. Since $\lvert \Delta \rvert < 9$ implies that $x_1, x_2 \in \Q$, we are done in this case.

Now for $\epsilon = -1$. Suppose that $\Q(x_1, x_2) \neq \Q(x_1^a - x_2^a)$. Then $G \neq H$. So there exists $\sigma \in G$ such that $\sigma(x_1) \neq x_1$. Suppose that $\sigma(x_2) = x_1$. Then the equality
\[ x_1^a - x_2^a = \sigma(x_1)^a - \sigma(x_2)^a\]
implies that
\[ 2 x_1^a = x_2^a + \sigma(x_1)^a.\]
Since $x_1$ is dominant and neither of $x_2, \sigma(x_1)$ is dominant, the absolute value of the right hand side of this equation is at most $2 (0.1 \lvert x_1 \rvert)^a$ by Proposition~\ref{prop:dom}. Clearly, this cannot happen. So $\sigma(x_2) \neq x_1$. 

In particular, $\sigma(x_2)$ is not dominant. Thus, by \eqref{eq:ineq1},
\begin{align}\label{eq:ineq4}
	\lvert x_1^a - x_2^a \rvert &= \lvert \sigma(x_1)^a - \sigma(x_2)^a \rvert \nonumber\\
	&\leq 2 (e^{\pi \lvert \Delta \rvert^{1/2}/2} + 2079)^a.
\end{align}
The bounds \eqref{eq:ineq2} and \eqref{eq:ineq4} are incompatible once 
\[ (e^{\pi \lvert \Delta \rvert^{1/2}} - 2079)^a > 3 (e^{\pi \lvert \Delta \rvert^{1/2}/2} + 2079)^a,\]
and in particular for all $\lvert \Delta \rvert \geq 9$. As above, this completes the proof in this case, since $\lvert \Delta \rvert < 9$ implies that $x_1, x_2 \in \Q$.

\subsubsection{The case where $\Delta_1 \neq \Delta_2$.}
Without loss of generality, we may assume that $\lvert \Delta_1 \rvert > \lvert \Delta_2 \rvert$. Suppose that $\Q(x_1, x_2) \neq \Q(x_1^a + \epsilon x_2^a)$. Then $G \neq H$. So there exists $\sigma \in G$ such that $\sigma(x_1) \neq x_1$ and $\sigma(x_2) \neq x_2$. Since
\[ x_1^a + \epsilon x_2^a = \sigma(x_1)^a + \epsilon \sigma(x_2)^a,\]
we obtain that
\[ \Q(x_1^a - \sigma(x_1)^a) = \Q(x_2^a - \sigma(x_2)^a).\]
Thus, by the already established equal discriminant case,
\[ \Q(x_1, \sigma(x_1)) = \Q(x_2, \sigma(x_2)).\]

Suppose first that $D_1 \neq D_2$. Then, by \cite[Corollary~3.3]{FayeRiffaut18}, $\Q(x_1)=\Q(x_2)$, and so $\Delta_1, \Delta_2$ are listed in \cite[Table~2]{AllombertBiluMadariaga15} by Lemma~\ref{lem:samefield}. We may assume that $x_1$ is dominant. So $\sigma(x_1)$ is not dominant. 

Observe that
\[1+ \epsilon \Big (\frac{x_2}{x_1}\Big)^a= \Big(\frac{\sigma(x_1)}{x_1}\Big)^a + \epsilon \Big(\frac{\sigma(x_2)}{x_1}\Big)^a.\]
Note that
\[ \Big\lvert 1+ \epsilon \Big(\frac{x_2}{x_1}\Big)^a \Big\rvert \geq 1 - \frac{\lvert x_2 \rvert}{\lvert x_1 \rvert},\]
since $\lvert x_2 \rvert < \lvert x_1 \rvert$. Further,
\[ \Big\lvert \Big(\frac{\sigma(x_1)}{x_1}\Big)^a + \epsilon \Big(\frac{\sigma(x_2)}{x_1}\Big)^a \Big\rvert \leq \frac{\lvert \sigma(x_1) \rvert}{\lvert x_1 \rvert} + \frac{\lvert \sigma(x_2) \rvert}{\lvert x_1 \rvert}, \]
since $\lvert \sigma(x_1) \rvert, \lvert \sigma(x_2) \rvert < \lvert x_1 \rvert$.

Write $x_2'$ for the dominant singular modulus of discriminant $\Delta_2$. Let $x_1'$ be a (not necessarily unique) singular modulus of discriminant $\Delta_1$ such that the only singular modulus of discriminant $\Delta_1$ with strictly larger absolute value than $x_1'$ is the dominant singular modulus $x_1$. (Such a singular modulus must exist since by assumption $x_1 \notin \Q$.) To complete the proof in this case, it suffices to verify that
\[ 1- \frac{\lvert x_2' \rvert}{\lvert x_1 \rvert} > \frac{\lvert x_1' \rvert}{\lvert x_1 \rvert} + \frac{\lvert x_2' \rvert}{\lvert x_1 \rvert}\]
for every ordered pair $(\Delta_1, \Delta_2)$ such that: $\Delta_1, \Delta_2$ appear in the same row of \cite[Table~2]{AllombertBiluMadariaga15} with $\lvert \Delta_1 \rvert > \lvert \Delta_2 \rvert$ and $h(\Delta_1) > 1$.
This may be done by a routine computation in PARI.

Now suppose that $D_1 = D_2$. In this case (see e.g. \cite[Lemma~7.2]{BiluFayeZhu19}), the identity $\Q(x_1, \sigma(x_1)) = \Q(x_2, \sigma(x_2))$ implies that $\Delta_1 = 4 \Delta_2$. Write $\Delta = \Delta_2$. As usual, take $x_1$ to be dominant. Then
\[ \lvert x_1^a + \epsilon x_2^a \rvert \geq (e^{2 \pi \lvert \Delta \rvert^{1/2}} - 2079)^a - (e^{\pi \lvert \Delta \rvert^{1/2}} + 2079)^a.\]
We also have that
\begin{align*}
	\lvert x_1^a + \epsilon x_2^a \rvert &= \lvert \sigma(x_1)^a + \epsilon \sigma(x_2)^a \rvert\\
	&\leq 2(e^{\pi \lvert \Delta \rvert^{1/2}} + 2079)^a,
\end{align*}
since $\sigma(x_1) \neq x_1$. These upper and lower bounds are incompatible once 
\[ (e^{2 \pi \lvert \Delta \rvert^{1/2}} - 2079)^a > 3 (e^{\pi \lvert \Delta \rvert^{1/2}} + 2079)^a.\]
In particular, they are incompatible when
\[ \frac{e^{2 \pi \lvert \Delta \rvert^{1/2}} - 2079}{e^{\pi \lvert \Delta \rvert^{1/2}} + 2079} > 3,\]
which happens for all $\lvert \Delta \rvert \geq 3$, i.e. for all possible $\Delta$.

\subsection{Proof of Theorem~\ref{thm:fieldprodpower}}

Let $x_1, x_2$ be distinct non-zero singular moduli of discriminants $\Delta_1, \Delta_2$ respectively. Let $\epsilon \in \{\pm 1\}$ and $a \in \Z \setminus \{0\}$. Clearly, it is enough to consider $a \geq 1$. Observe that Proposition~\ref{prop:fieldpower} implies that
\[ \Q(x_1, x_2) = \Q(x_1^a, x_2^a).\]
Let $L$ be the Galois closure of $\Q(x_1, x_2) / \Q$. Define
\[G= \Gal(L / \Q(x_1^a x_2^{\epsilon a}))\]
and
\[ H = \Gal(L/ \Q(x_1, x_2)).\]
So $H \leq G$, and $H=G$ if and only if $\Q(x_1, x_2) = \Q(x_1^a x_2^{\epsilon a})$. By analogous arguments to those in the proof of Theorem~\ref{thm:fieldsumpower}, we have that
\begin{align*}
	H &= \{ \sigma \in G : \sigma(x_1^a)=x_1^a\}= \{ \sigma \in G : \sigma(x_1)=x_1\}\\
	&= \{ \sigma \in G : \sigma(x_2^a)=x_2^a\}= \{ \sigma \in G : \sigma(x_2)=x_2\}.
\end{align*}

One may now prove Theorem~\ref{thm:fieldprodpower} by following the proof of \cite[Theorem~5.1]{FayeRiffaut18}, thanks to the elementary fact that
\[ \lvert (xy)^a \rvert < \lvert (wz)^a \rvert \mbox{ if and only if } \lvert xy \rvert < \lvert wz \rvert.\]

Suppose first that $\Delta_1 = \Delta_2$ and $\epsilon =1$. Write $\Delta = \Delta_1 = \Delta_2$. In this case, suppose that $[\Q(x_1, x_2) : \Q(x_1^a x_2^a)] > 2$. Then $[G : H] > 2$. So there exists $\sigma \in G$ such that $\sigma(x_1^a) \neq x_1^a$ and $\sigma(x_1^a) \neq x_2^a$. We may assume that $x_1$ is dominant. Since $x_1^a x_2^a = \sigma(x_1^a) \sigma(x_2^a)$, we have that $\sigma(x_2^a) \neq x_1^a$. In particular, $\sigma(x_1), \sigma(x_2) \neq x_1$, and so neither of $\sigma(x_1), \sigma(x_2)$ is dominant. We therefore obtain, exactly as in \cite[\S5.1.1]{FayeRiffaut18}, that $\lvert \Delta \rvert \leq 395$.

To deal with the case where $\lvert \Delta \rvert \leq 395$, we do the following. Since $x_1^a x_2^a = \sigma(x_1^a) \sigma(x_2^a)$, we have that
\[ \frac{x_1 x_2}{\sigma(x_1) \sigma(x_2)}\]
is a root of unity, which is an element of the field $L$. For each possible choice of singular modulus $x_2$ of discriminant $\Delta$ such that $x_2 \neq x_1$ and for each conjugate $(x_1', x_2')$ of $(x_1, x_2)$ such that $x_1', x_2' \neq x_1$, we may verify in PARI whether 
\[ \frac{x_1 x_2}{x_1' x_2'}\]
is a root of unity in $L$. If none are roots of unity, then we may eliminate the corresponding discriminant $\Delta$. In this way, we complete the proof of Theorem~\ref{thm:fieldprodpower} in this case.

Now suppose that $\Delta_1 = \Delta_2$ and $\epsilon =-1$. Write $\Delta = \Delta_1 = \Delta_2$. In this case, suppose that $\Q(x_1, x_2) \neq \Q(x_1^a x_2^a)$. In this case, we obtain, exactly as in \cite[\S5.1.2]{FayeRiffaut18}, that $\lvert \Delta \rvert \leq 395$. These small values of $\Delta$ may then be handled in an analogous way to that described above.

The cases where $\Delta_1 \neq \Delta_2$ are then treated similarly. We bound $\Delta_1, \Delta_2$ by the same argument as in \cite[\S5]{FayeRiffaut18}, and then eliminate these remaining small discriminants by the process described above. This completes the proof of Theorem~\ref{thm:fieldprodpower}.

\section{The $\epsilon_1=\epsilon_2=\epsilon_3$ case of Theorem~\ref{thm:sum}}\label{sec:sum}

We now begin the proof of Theorem~\ref{thm:sum}. The ``if'' direction is immediate. We will prove the ``only if''. In this section, we prove the ``only if'' direction of Theorem~\ref{thm:sum} in the case that $\epsilon_1=\epsilon_2=\epsilon_3$.  

Suppose then that $x_1, x_2, x_3$ are singular moduli such that 
\[ \epsilon_1 x_1^a + \epsilon_2 x_2^a + \epsilon_3 x_3^a \in \Q\]
for some $a \in \Z_{>0}$ and $\epsilon_1=\epsilon_2=\epsilon_3 \in \{\pm 1\}$. Clearly, we may assume that $\epsilon_1=\epsilon_2=\epsilon_3=1$. Write $\Delta_i$ for the discriminant of $x_i$ and $h_i = h(\Delta_i)$ for the corresponding class number. We will show that we must be in one of cases (1), (3), or (4) of Theorem~\ref{thm:sum}.

If $x_1 = x_2 = x_3$, then, by Proposition~\ref{prop:fieldpower}, $x_1, x_2, x_3 \in \Q$. If $x_i = x_j \neq x_k$, then $2 x_i^a + x_k^a \in \Q$. Then, by Theorem~\ref{thm:sum2}, either $x_1, x_2, x_3 \in \Q$ or $\Q(x_i) = \Q(x_k)$ and this field has degree $2$ over $\Q$. In the first case, we are done. 

We now show that the second case cannot occur. Suppose then that $\Q(x_i) = \Q(x_k)$ and this field has degree $2$ over $\Q$. If $\Delta_i = \Delta_k$, then $x_i^a + x_k^a \in \Q$, since $x_i, x_k$ are the two roots of the polynomial $H_{\Delta_i}$. But $x_i^a \notin \Q$ by Proposition~\ref{prop:fieldpower}. Hence, $2 x_i^a + x_k^a \notin \Q$, a contradiction. 

So we must have that $\Delta_i \neq \Delta_k$. Since $x_i, x_k$ are both degree $2$, they each have a unique conjugate, which we denote $x_i', x_k'$ respectively. Since $\Q(x_i) = \Q(x_k)$, one has that 
\[ 2 x_i^a + x_k^a = 2 (x_i')^a + (x_k')^a.\]

Suppose first that $\lvert \Delta_i \rvert > \lvert \Delta_k \rvert$. Without loss of generality, we may assume that $x_i$ is dominant. We have that
\[ 2 = 2 \Big(\frac{x_i'}{x_i}\Big)^a + \Big(\frac{x_k'}{x_i}\Big)^a - \Big(\frac{x_k}{x_i}\Big)^a.\]
 By Proposition~\ref{prop:dominc}, $\lvert x_i \rvert > \lvert x_i' \rvert, \lvert x_k \rvert, \lvert x_k' \rvert$. Thus, the absolute value of the right hand side of the above equation must be
\[\leq 2 \frac{\lvert x_i' \rvert}{\lvert x_i \rvert} + \frac{\lvert x_k' \rvert}{\lvert x_i \rvert} + \frac{\lvert x_k \rvert}{\lvert x_i \rvert}.\]
We may however verify in PARI that this expression is always $< 2$, a contradiction.

Suppose then that $\lvert \Delta_i \rvert < \lvert \Delta_k \rvert$. In this case, we may, without loss of generality, assume that $x_k$ is dominant. Then
\[ 1 = 2 \Big(\frac{x_i'}{x_k}\Big)^a - 2 \Big(\frac{x_i}{x_k}\Big)^a + \Big(\frac{x_k'}{x_k}\Big)^a.\]
By Proposition~\ref{prop:dominc}, $\lvert x_k \rvert > \lvert x_i \rvert, \lvert x_i' \rvert, \lvert x_k' \rvert$. Thus, the absolute value of the right hand side of the above equation must be
\[\leq 2 \frac{\lvert x_i' \rvert}{\lvert x_k \rvert} + 2 \frac{\lvert x_i \rvert}{\lvert x_k \rvert} + \frac{\lvert x_k' \rvert}{\lvert x_k \rvert}.\]
We may however verify in PARI that this expression is always $< 1$, a contradiction. We have thus shown that the second case cannot occur.

Hence, we may assume from now on that $x_1, x_2, x_3$ are pairwise distinct. Suppose next that some $x_i \in \Q$. Then $x_j^a + x_k^a \in \Q$. By Theorem~\ref{thm:fieldsumpower}, if $\Delta_j \neq \Delta_k$, then
\[\Q = \Q( x_j^a + x_k^a) = \Q(x_j, x_k)\]
and so $x_j, x_k \in \Q$. If $\Delta_j = \Delta_k$, then $x_j, x_k$ are distinct, conjugate singular moduli and, again by Theorem~\ref{thm:fieldsumpower}, $[\Q(x_j, x_k) : \Q] \leq 2$; we thus must have that $x_j, x_k$ are distinct conjugate singular moduli of degree $2$. Thus, we are in case (3) of the theorem. So subsequently we assume also that no $x_i$ is rational. In particular, $h_i \geq 2$ for all $i$.

Without loss of generality, we may assume that $h_1 \geq h_2 \geq h_3$. Let $A \in \Q$ be such that
\[ x_1^a + x_2^a + x_3^a = A. \]
Observe that $\Q(x_1) = \Q(x_1^a) = \Q(x_2^a + x_3^a)$, where the first equality follows from Proposition~\ref{prop:fieldpower}. Thus,
by Theorem~\ref{thm:fieldsumpower}, we have that
\[[\Q(x_1) : \Q] = [\Q(x_2^a + x_3^a) : \Q]=  \begin{cases} \mbox{ either } [\Q(x_2, x_3) : \Q], \\
	\mbox{ or } \frac{1}{2}[\Q(x_2, x_3) : \Q].
\end{cases}\]
We now argue as in \cite{Fowler20}. Since $h_2 = [\Q(x_2) : \Q]$ and $h_3 = [\Q(x_3) : \Q]$ each divide $[\Q(x_2,x_3) : \Q]$, we obtain that $h_2, h_3 \mid 2 [\Q(x_1) : \Q]$. So $h_2, h_3 \mid 2 h_1$. Symmetrically, $h_1, h_2 \mid 2h_3$ and $h_1, h_3 \mid 2 h_2$.  By assumption, $h_1 \geq h_2 \geq h_3$. Therefore, one of the following holds:
\begin{enumerate}
	\item $h_1 = h_2 = h_3$,
	\item $h_1 = h_2 = 2h_3$,
	\item $h_1 = 2 h_2 = 2 h_3$.
\end{enumerate} 
We will consider each of these cases in turn.

 \subsection{The case where $h_1 = h_2 = h_3$}

Suppose first that $h_1 = h_2 = h_3$. Write $h$ for this class number. We make another case distinction.

\subsubsection{The subcase where $\Delta_1 = \Delta_2 = \Delta_3$}\label{subsubsec0}
Suppose that $\Delta_1 = \Delta_2 = \Delta_3$. Write $\Delta$ for this common discriminant. Note that we must have that $h \geq 3$, since the $x_i$ are pairwise distinct. If $h=3$, then this is case (4) of Theorem~\ref{thm:sum}. 

So assume that $h \geq 4$. Taking conjugates as necessary, we may assume that $x_1$ is dominant. So $x_2, x_3$ are not dominant. Thus, by \eqref{eq:ineq1},
\begin{align}\label{eq:ineq5}
\lvert A \rvert \geq (e^{\pi \lvert \Delta \rvert^{1/2}} -2079)^a - 2 (e^{\pi \lvert \Delta \rvert^{1/2}/2} + 2079)^a.
\end{align} 
Since $h \geq 4$, there is a conjugate $(x_1', x_2', x_3')$ of $(x_1, x_2, x_3)$ such that none of $x_1', x_2', x_3'$ is dominant. Therefore, by \eqref{eq:ineq1},
\begin{align}\label{eq:ineq6}
	\lvert A \rvert &= \lvert (x_1')^a + (x_2')^a + (x_3')^a \rvert \nonumber\\
	&\leq 3 (e^{\pi \lvert \Delta \rvert^{1/2}/2} + 2079)^a.
\end{align}
(In fact, a better bound holds, since at least one of $x_1', x_2', x_3'$ is not subdominant, because there are at most two subdominant singular moduli of discriminant $\Delta$. For our purposes, however, the given bound suffices.) 

In particular, the two bounds \eqref{eq:ineq5} and \eqref{eq:ineq6} are incompatible whenever
\[ (e^{\pi \lvert \Delta \rvert^{1/2}} -2079)^a > 5 (e^{\pi \lvert \Delta \rvert^{1/2}/2} + 2079)^a,\]
and so certainly whenever
\[ \frac{e^{\pi \lvert \Delta \rvert^{1/2}} -2079}{e^{\pi \lvert \Delta \rvert^{1/2}/2} + 2079} > 5,\]
which happens if $\lvert \Delta \rvert \geq 10$. Since $\lvert \Delta \rvert < 10$ would contradict the assumption that $h \geq 4$, we are done in this case.

\subsubsection{The subcase where the $\Delta_i$ are not all equal} \label{subsubsec1}
Next, suppose that the $\Delta_i$ are not all equal. Without loss of generality, we may assume that $\lvert \Delta_1 \rvert > \lvert \Delta_2 \rvert$ and $\lvert \Delta_1 \rvert \geq \lvert \Delta_3 \rvert \geq \lvert \Delta_2 \rvert$. Then, by Proposition~\ref{prop:fieldpower} and Theorem~\ref{thm:fieldsumpower},
\[ \Q(x_3) = \Q(x_3^a)= \Q(x_1^a + x_2^a) = \Q(x_1, x_2),\]
since $\Delta_1 \neq \Delta_2$. So
\[ \Q(x_1), \Q(x_2) \subset \Q(x_3).\]
Since $h_1 = h_2 = h_3$, we must then have that
\begin{align}\label{eq:1}
\Q(x_1) = \Q(x_2) = \Q(x_3).
\end{align}
If $\Q(\sqrt{\Delta_i}) \neq \Q(\sqrt{\Delta_j})$ for some $i \neq j$, then all possibilities for $(\Delta_1, \Delta_2, \Delta_3)$ are known by Lemma~\ref{lem:samefield}.

We now explain how we eliminate these possibilities. For now, assume that $h = 2$. Then, by \eqref{eq:1}, there is a unique non-trivial Galois conjugate $(x_1', x_2', x_3')$ of $(x_1, x_2, x_3)$ and this satisfies $x_i' \neq x_i$ for all $i$. One has that
\[(x_1')^a + (x_2')^a + (x_3')^a = A.\]
If $\Delta_1 = \Delta_3$, then $x_1, x_3$ are distinct, conjugate singular moduli of degree $2$, so one has that $x_1 = x_3'$ and $x_3 = x_1'$. Thus, in this case, one obtains that $(x_2)^a - (x_2')^a=0$, but this contradicts Theorem~\ref{thm:prod2}. Hence, $\Delta_1 \neq \Delta_3$. Similarly, one shows that $\Delta_2 \neq \Delta_3$. So we must have that $\lvert \Delta_1 \rvert > \lvert \Delta_3 \rvert > \lvert \Delta_2 \rvert$.

Observe that
\[ x_1^a = (x_1')^a -x_2^a + (x_2')^a - x_3^a + (x_3')^a,\]
and thus
\[ 1 = \Big(\frac{x_1'}{x_1}\Big)^a -\Big(\frac{x_2}{x_1}\Big)^a +\Big(\frac{x_2'}{x_1}\Big)^a -\Big(\frac{x_3}{x_1}\Big)^a +\Big(\frac{x_3'}{x_1}\Big)^a.\]
We may assume that $x_1$ is dominant. Thus, $\lvert x_1' \rvert, \lvert x_2 \rvert, \lvert x_2' \rvert, \lvert x_3 \rvert, \lvert x_3' \rvert < \lvert x_1 \rvert$ by Proposition~\ref{prop:dominc}, since $\lvert \Delta_1 \rvert > \lvert \Delta_2 \rvert, \lvert \Delta_3 \rvert$ and $x_1'$ is not dominant. Hence,
\begin{align*}
1 =&	\Big\lvert \Big(\frac{x_1'}{x_1}\Big)^a -\Big(\frac{x_2}{x_1}\Big)^a +(\frac{x_2'}{x_1})^a -(\frac{x_3}{x_1})^a +\Big(\frac{x_3'}{x_1}\Big)^a \Big\rvert\\
\leq &\Big\lvert \Big(\frac{x_1'}{x_1}\Big) \Big\rvert + \Big\lvert \Big(\frac{x_2}{x_1}\Big) \Big\rvert +\Big\lvert \Big(\frac{x_2'}{x_1}\Big) \Big\rvert +\Big\lvert \Big(\frac{x_3}{x_1}\Big) \Big\rvert +\Big\lvert \Big(\frac{x_3'}{x_1}\Big) \Big\rvert.
\end{align*}
However, we may verify in PARI that this last expression is $<1$ for every possible choice of $x_2, x_3$ for each of the relevant triples $(\Delta_1, \Delta_2, \Delta_3)$. We may thereby eliminate each of these possible $(\Delta_1, \Delta_2, \Delta_3)$.

We now assume that $h>2$. Then inspection of \cite[Table~2]{AllombertBiluMadariaga15} gives us that $h \geq 4$. Assume that $x_1$ is dominant. Since $h \geq 4$, there exists a conjugate $(x_1', x_2', x_3')$ of $(x_1, x_2, x_3)$ such that no $x_i'$ is dominant. Proceeding as above, we have that 
\[ 1 = \Big(\frac{x_1'}{x_1}\Big)^a -\Big(\frac{x_2}{x_1}\Big)^a +\Big(\frac{x_2'}{x_1}\Big)^a -\Big(\frac{x_3}{x_1}\Big)^a +\Big(\frac{x_3'}{x_1}\Big)^a.\]
By Proposition~\ref{prop:dominc}, $\lvert x_1' \rvert, \lvert x_2 \rvert, \lvert x_2' \rvert, \lvert x_3 \rvert, \lvert x_3' \rvert < \lvert x_1 \rvert$, since: $\lvert \Delta_1 \rvert > \lvert \Delta_2 \rvert$; $\lvert \Delta_1 \rvert \geq \lvert \Delta_3 \rvert$; no $x_i'$ is dominant; and if $\Delta_3 = \Delta_1$, then $x_3$ is not dominant (because $x_3 \neq x_1$). Consequently,
\begin{align*}
1=	&	\Big\lvert \Big(\frac{x_1'}{x_1}\Big)^a -\Big(\frac{x_2}{x_1}\Big)^a +\Big(\frac{x_2'}{x_1}\Big)^a -\Big(\frac{x_3}{x_1}\Big)^a +\Big(\frac{x_3'}{x_1}\Big)^a \rvert\\
	\leq &\Big\lvert \Big(\frac{x_1'}{x_1}\Big) \Big\rvert + \Big\lvert \Big(\frac{x_2}{x_1}\Big) \Big\rvert +\Big\lvert \Big(\frac{x_2'}{x_1}\Big) \Big\rvert +\Big\lvert \Big(\frac{x_3}{x_1}\Big) \Big\rvert +\Big\lvert \Big(\frac{x_3'}{x_1}\Big) \Big\rvert.
\end{align*}
Once again, we may verify in PARI that this last expression is $<1$ for every possible choice of $x_2, x_3$ for each of the relevant triples $(\Delta_1, \Delta_2, \Delta_3)$. We may thereby eliminate this case.

Now suppose instead that $\Q(\sqrt{\Delta_1}) = \Q(\sqrt{\Delta_2})= \Q(\sqrt{\Delta_3})$. Then Lemma~\ref{lem:samefield} implies that $\Delta_1 = 4 \Delta_2$ and $\Delta_3 \in \{\Delta_1, \Delta_2\}$. Write $\Delta = \Delta_2$.

Suppose first that $\Delta_1 = \Delta_3 = 4 \Delta$ and $\Delta_2 = \Delta$. Taking conjugates, we may assume that $x_1$ is dominant. So $x_3$ is not dominant. Then, by \eqref{eq:ineq1},
\begin{align}\label{eq:ineq7}
	\lvert A \rvert &\geq (e^{2 \pi \lvert \Delta \rvert^{1/2}} -2079)^a - (e^{\pi \lvert \Delta \rvert^{1/2}} + 2079)^a - (e^{2 \pi \lvert \Delta \rvert^{1/2}/2} + 2079)^a \nonumber\\
	& \geq (e^{2 \pi \lvert \Delta \rvert^{1/2}} -2079)^a - 2 (e^{\pi \lvert \Delta \rvert^{1/2}} + 2079)^a.
\end{align}
If $h = 2$, then $x_1, x_3$ are two distinct, conjugate singular moduli of degree $2$, so $x_1^a + x_3^a \in \Q$ and hence $x_2 \in \Q$ (since $\Q(x_2^a) = \Q(x_2)$ by Proposition~\ref{prop:fieldpower}). This cannot happen, since $h = 2$. So we must have that $h \geq 3$. Consequently, there is a conjugate $(x_1', x_2', x_3')$ with neither $x_1', x_3'$ dominant. Such a conjugate gives rise to the upper bound
\begin{align}\label{eq:ineq8}
	 \lvert A \rvert &\leq 2 (e^{2 \pi \lvert \Delta \rvert^{1/2}/2} +2079)^a + (e^{\pi \lvert \Delta \rvert^{1/2}} + 2079)^a \nonumber \\
	 &\leq 3 (e^{\pi \lvert \Delta \rvert^{1/2}} + 2079)^a.
	 \end{align}
In particular, the bounds \eqref{eq:ineq7} and \eqref{eq:ineq8} are incompatible whenever
\[ (e^{2 \pi \lvert \Delta \rvert^{1/2}} -2079)^a > 5 (e^{\pi \lvert \Delta \rvert^{1/2}} + 2079)^a,\]
and so certainly whenever
\[ \frac{e^{2 \pi \lvert \Delta \rvert^{1/2}} -2079}{e^{\pi \lvert \Delta \rvert^{1/2}} + 2079} > 5,\]
which happens if $\lvert \Delta \rvert \geq 3$, i.e. for all discriminants $\Delta$. So this case is proved.

Suppose next that $\Delta_1 = 4 \Delta$ and $\Delta_2 = \Delta_3 = \Delta$. Taking conjugates, we may assume that $x_1$ is dominant. At most one of $x_2, x_3$ is dominant, so
\begin{align*}
	 \lvert A \rvert &\geq (e^{2 \pi \lvert \Delta \rvert^{1/2}} -2079)^a - (e^{\pi \lvert \Delta \rvert^{1/2}} + 2079)^a - (e^{\pi \lvert \Delta \rvert^{1/2}/2} + 2079)^a \nonumber\\
	 &\geq (e^{2 \pi \lvert \Delta \rvert^{1/2}} -2079)^a - 2 (e^{\pi \lvert \Delta \rvert^{1/2}} + 2079)^a.
	 \end{align*}
Since $h \geq 2$, there is a conjugate $(x_1', x_2', x_3')$ with $x_1'$ not dominant. This conjugate gives rise to the upper bound
\begin{align*}
	 \lvert A \rvert &\leq (e^{2 \pi \lvert \Delta \rvert^{1/2}/3} +2079)^a + (e^{\pi \lvert \Delta \rvert^{1/2}} + 2079)^a + (e^{\pi \lvert \Delta \rvert^{1/2}/2} + 2079)^a\\
	 &\leq 3 (e^{\pi \lvert \Delta \rvert^{1/2}} + 2079)^a.
	 \end{align*}
We thus complete the proof of this case exactly as in the previous case.

\subsection{The case where $h_1 = h_2 = 2 h_3$}
Now we come to the case that $h_1 = h_2 = 2 h_3$. In particular, $\Delta_2 \neq \Delta_3$. Thus, by Theorem~\ref{thm:fieldsumpower},
\[ \Q(x_1) = \Q(x_1^a)= \Q(x_2^a + x_3^a) = \Q(x_2, x_3) \supset \Q(x_2), \Q(x_3).\]
And so
\[ \Q(x_1) = \Q(x_2) \supsetneq \Q(x_3)\]
with $[\Q(x_1) : \Q(x_3)]=2$.
If $\Delta_1 \neq \Delta_2$, then 
\[\Q(x_3) = \Q(x_3^{a})= \Q(x_1^a + x_2^a) = \Q(x_1, x_2) = \Q(x_1),\]
a contradiction. So $\Delta_1 = \Delta_2$. Observe also that $h_1 = h_2 \geq 4$ since $h_3 \geq 2$.

\subsubsection{The subcase where $\Q(\sqrt{\Delta_1})=\Q(\sqrt{\Delta_3})$} Suppose that $\Q(\sqrt{\Delta_1})=\Q(\sqrt{\Delta_3})$. Then Lemma~\ref{lem:subfieldsamefund} implies that $\Delta_1 \in \{9 \Delta_3/4, 4 \Delta_3, 9 \Delta_3, 16 \Delta_3\}$. 

First, suppose that $\Delta_1 = \Delta_2 = 9 \Delta/4$ and $\Delta_3 = \Delta$. We may assume that $x_1$ is dominant, and hence $x_2$ is not dominant. We obtain that
\begin{align}\label{eq:ineq9}
	 \lvert A \rvert &\geq (e^{3 \pi \lvert \Delta \rvert^{1/2}/2} -2079)^a - (e^{3 \pi \lvert \Delta \rvert^{1/2}/4} + 2079)^a - (e^{\pi \lvert \Delta \rvert^{1/2}} + 2079)^a \nonumber\\
	 &\geq (e^{3 \pi \lvert \Delta \rvert^{1/2}/2} -2079)^a - 2 (e^{\pi \lvert \Delta \rvert^{1/2}} + 2079)^a.
	 \end{align}
Since $h_1 \geq 4$, there exists a conjugate $(x_1', x_2', x_3')$ with $x_1', x_2'$ both not dominant. Thus,
\begin{align}\label{eq:ineq10}
	 \lvert A \rvert &\leq 2 (e^{3 \pi \lvert \Delta \rvert^{1/2}/4} +2079)^a + (e^{\pi \lvert \Delta \rvert^{1/2}} + 2079)^a \nonumber\\
	&\leq  3 (e^{\pi \lvert \Delta \rvert^{1/2}} + 2079)^a.
	 \end{align}
 In particular, the bounds \eqref{eq:ineq9} and \eqref{eq:ineq10} are incompatible whenever
 \[ (e^{3 \pi \lvert \Delta \rvert^{1/2} / 2} -2079)^a > 5 (e^{\pi \lvert \Delta \rvert^{1/2}} + 2079)^a,\]
 and so certainly whenever
 \[ \frac{e^{3 \pi \lvert \Delta \rvert^{1/2} /2} -2079}{e^{\pi \lvert \Delta \rvert^{1/2}} + 2079} > 5,\]
 which happens if $\lvert \Delta \rvert \geq 5$. Since $\lvert \Delta \rvert < 5$ would contradict the assumption that $h_3 \geq 2$, we are done in this case.

Second, suppose that $\Delta_1 = \Delta_2 = 4 \Delta$ and $\Delta_3 = \Delta$. Then, taking $x_1$ dominant,
\begin{align}\label{eq:ineq11}
	 \lvert A \rvert &\geq (e^{2 \pi \lvert \Delta \rvert^{1/2}} -2079)^a - 2 (e^{ \pi \lvert \Delta \rvert^{1/2}} + 2079)^a.
	 \end{align}
Since $h_1 \geq 4$, there exists a conjugate $(x_1', x_2', x_3')$ with $x_1', x_2'$ both not dominant. Thus,
\begin{align} \label{eq:ineq12}
\lvert A \rvert \leq 3 (e^{ \pi \lvert \Delta \rvert^{1/2}} +2079)^a.
\end{align}
The bounds \eqref{eq:ineq11} and \eqref{eq:ineq12} are incompatible whenever
\[ (e^{2 \pi \lvert \Delta \rvert^{1/2}} -2079)^a > 5 (e^{\pi \lvert \Delta \rvert^{1/2}} + 2079)^a,\]
which we have already seen is true for all discriminants $\Delta$.

Third, suppose that $\Delta_1 = \Delta_2 = 9 \Delta$ and $\Delta_3 = \Delta$. Then, taking $x_1$ dominant,
\begin{align}\label{eq:ineq13}
	 \lvert A \rvert &\geq (e^{3 \pi \lvert \Delta \rvert^{1/2}} -2079)^a - (e^{3 \pi \lvert \Delta \rvert^{1/2}/2} + 2079)^a - (e^{\pi \lvert \Delta \rvert^{1/2}} + 2079)^a \nonumber\\
	 &\geq (e^{3 \pi \lvert \Delta \rvert^{1/2}} -2079)^a - 2 (e^{3 \pi \lvert \Delta \rvert^{1/2}/2} + 2079)^a.
	 \end{align}
Since $h_1 \geq 4$, there exists a conjugate $(x_1', x_2', x_3')$ with $x_1', x_2'$ both not dominant. Thus,
\begin{align}\label{eq:ineq14}
	 \lvert A \rvert &\leq 2 (e^{3 \pi \lvert \Delta \rvert^{1/2}/2} +2079)^a + (e^{\pi \lvert \Delta \rvert^{1/2}} + 2079)^a \nonumber \\
	 &\leq 3(e^{3 \pi \lvert \Delta \rvert^{1/2}/2} +2079)^a .
	 \end{align}
The bounds \eqref{eq:ineq13} and \eqref{eq:ineq14} are incompatible whenever
\[ (e^{3 \pi \lvert \Delta \rvert^{1/2}} -2079)^a > 5 (e^{3 \pi \lvert \Delta \rvert^{1/2}/2} + 2079)^a,\]
and so certainly whenever
\[\frac{e^{3 \pi \lvert \Delta \rvert^{1/2}} -2079}{e^{3 \pi \lvert \Delta \rvert^{1/2}/2} + 2079} > 5,\]
which holds for all $\lvert \Delta \rvert \geq 2$, i.e. for all discriminants $\Delta$. 

Finally suppose that $\Delta_1 = \Delta_2 = 16 \Delta$ and $\Delta_3 = \Delta$. Then, taking $x_1$ dominant,
\begin{align}\label{eq:ineq15}
	\lvert A \rvert &\geq (e^{4 \pi \lvert \Delta \rvert^{1/2}} -2079)^a - (e^{2 \pi \lvert \Delta \rvert^{1/2}} + 2079)^a - (e^{\pi \lvert \Delta \rvert^{1/2}} + 2079)^a\nonumber \\
	&\geq (e^{4 \pi \lvert \Delta \rvert^{1/2}} -2079)^a - 2(e^{2 \pi \lvert \Delta \rvert^{1/2}} + 2079)^a.
	\end{align}
Since $h_1 \geq 4$, there exists a conjugate $(x_1', x_2', x_3')$ with $x_1', x_2'$ both not dominant. Thus,
\begin{align}\label{eq:ineq16}
	\lvert A \rvert &\leq 2 (e^{2 \pi \lvert \Delta \rvert^{1/2}} +2079)^a + (e^{\pi \lvert \Delta \rvert^{1/2}} + 2079)^a \nonumber\\
	&\leq 3 (e^{2 \pi \lvert \Delta \rvert^{1/2}} +2079)^a.
\end{align}
The bounds \eqref{eq:ineq15} and \eqref{eq:ineq16} are incompatible whenever
\[ (e^{4 \pi \lvert \Delta \rvert^{1/2}} -2079)^a > 5 (e^{2 \pi \lvert \Delta \rvert^{1/2}} + 2079)^a,\]
and so certainly whenever
\[\frac{e^{4 \pi \lvert \Delta \rvert^{1/2}} -2079}{e^{2 \pi \lvert \Delta \rvert^{1/2}} + 2079} > 5,\]
which holds for all $\lvert \Delta \rvert \geq 1$, i.e. for all discriminants $\Delta$.

\subsubsection{The subcase where $\Q(\sqrt{\Delta_1}) \neq \Q(\sqrt{\Delta_3})$}
Suppose that $\Q(\sqrt{\Delta_1}) \neq \Q(\sqrt{\Delta_3})$. Then Lemma~\ref{lem:subfielddiffund} implies that either one of $\Delta_1, \Delta_3$ is listed and the corresponding extension $\Q(x_i) / \Q$ is Galois, or $h_1 \geq 128$. 

We begin with the first case. In this case, we may find all possibilities for $(\Delta_1, \Delta_2, \Delta_3)$. When $\Delta_1$ is listed, PARI finds\footnote{The command which does this, delta1\textunderscore listed\textunderscore triples(), is contained in the script general.gp.} 330 possible triples $(\Delta_1, \Delta_2, \Delta_3)$. All but $9$ of these have $\lvert \Delta_1 \rvert > \lvert \Delta_3 \rvert$. So for now assume that $\lvert \Delta_1 \rvert > \lvert \Delta_3 \rvert$.

Our approach is basically the same as previously. We may assume that $x_1$ is dominant, and so $x_2$ is not dominant. Since $x_1^a + x_2^a + x_3^a = A$, for every conjugate $(x_1', x_2', x_3')$ of $(x_1, x_2, x_3)$, we have that
\[ (x_1')^a + (x_2')^a + (x_3')^a = A\]
as well. Thus,
\[ 1 = \Big(\frac{x_1'}{x_1}\Big)^a -\Big(\frac{x_2}{x_1}\Big)^a +\Big(\frac{x_2'}{x_1}\Big)^a -\Big(\frac{x_3}{x_1}\Big)^a +\Big(\frac{x_3'}{x_1}\Big)^a.\]
Since $h_1 \geq 4$, we may always choose $(x_1', x_2', x_3')$ such that neither of $x_1', x_2'$ is dominant. Therefore, $\lvert x_1' \rvert, \lvert x_2 \rvert, \lvert x_2' \rvert, \lvert x_3 \rvert, \lvert x_3' \rvert < \lvert x_1 \rvert$ by Proposition~\ref{prop:dominc}, since $\lvert \Delta_1 \rvert > \lvert \Delta_3 \rvert$ and $x_1', x_2, x_2'$ are not dominant. Hence,
\begin{align*}
1=	&	\Big\lvert \Big(\frac{x_1'}{x_1}\Big)^a -\Big(\frac{x_2}{x_1}\Big)^a +\Big(\frac{x_2'}{x_1}\Big)^a -\Big(\frac{x_3}{x_1}\Big)^a +\Big(\frac{x_3'}{x_1}\Big)^a \Big\rvert\\
	\leq &\Big\lvert \Big(\frac{x_1'}{x_1}\Big) \Big\rvert + \Big\lvert \Big(\frac{x_2}{x_1}\Big) \Big\rvert +\Big\lvert \Big(\frac{x_2'}{x_1}\Big) \Big\rvert +\Big\lvert \Big(\frac{x_3}{x_1}\Big) \Big\rvert +\Big\lvert \Big(\frac{x_3'}{x_1}\Big) \Big\rvert.
\end{align*}
However, we may verify in PARI that this last expression is $<1$ for every possible choice of $x_2, x_3$ for all but one of the $321$ triples $(\Delta_1, \Delta_2, \Delta_3)$, and so we may eliminate those triples.

The one exception is the triple $(\Delta_1, \Delta_2, \Delta_3) = (-240, -240, -235)$, which has $h(\Delta_1) = 4$. The problem here is that, since $-240, -235$ are so close to one another, one has that
\[ \frac{\lvert x_3 \rvert}{\lvert x_1 \rvert} \approx 0.6\]
if $x_1, x_3$ are the dominant singular moduli of respective discriminants $-240, -235$. So the approach described above does not work, because $x_3, x_3'$ could both be dominant. However, we can avoid this issue in the following way. 

Recall that $\Q(x_3) \subset\Q(x_1)$ with $[\Q(x_1) : \Q(x_3)]=2$ and the extension $\Q(x_1) /\Q$ is Galois. Further, $x_1$ is dominant by assumption, and so $x_2$ is not dominant. Suppose first that $x_3$ is also dominant. Since $h_1=4$, there are exactly $4$ conjugates of $(x_1, x_2, x_3)$, including $(x_1, x_2, x_3)$ itself. Observe that $x_3$ occurs exactly twice among the third coordinates of these $4$ conjugates, and by assumption one of these occurrences is in $(x_1, x_2, x_3)$ itself. Note that $2$ of the $3$ non-trivial conjugates of $(x_1, x_2, x_3)$ have neither their first nor their second coordinate dominant. At most $1$ of these $2$ conjugates has its third coordinate dominant. Thus, there exists a conjugate $(x_1', x_2', x_3')$ of $(x_1, x_2, x_3)$ such that none of $x_1', x_2', x_3'$ are dominant.

Now suppose that $x_3$ is not dominant. Exactly $2$ of the $3$ non-trivial conjugates of $(x_1, x_2, x_3)$ have neither their first nor their second coordinate dominant. Further, exactly $2$ of the $3$ non-trivial conjugates of $(x_1, x_2, x_3)$ have their third coordinate dominant. Thus, there exists a conjugate $(x_1', x_2', x_3')$ of $(x_1, x_2, x_3)$ such that $x_3'$ is dominant and neither of $x_1', x_2'$ is dominant.

In either case, we may always find a conjugate $(x_1', x_2', x_3')$ of $(x_1, x_2, x_3)$ such that $x_1', x_2'$ are not dominant and at most one of $x_3, x_3'$ is dominant. For such $(x_1, x_2, x_3)$ and $(x_1', x_2', x_3')$, we may verify in PARI that
\[\Big\lvert \Big(\frac{x_1'}{x_1}\Big) \Big\rvert + \Big\lvert \Big(\frac{x_2}{x_1}\Big) \Big\rvert +\Big\lvert \Big(\frac{x_2'}{x_1}\Big) \Big\rvert +\Big\lvert \Big(\frac{x_3}{x_1}\Big) \Big\rvert +\Big\lvert \Big(\frac{x_3'}{x_1}\Big) \rvert <1.\]
We may thus eliminate the triple $(-240, -240, -235)$.

We now rule out the $9$ triples $(\Delta_1, \Delta_2, \Delta_3)$ with $\lvert \Delta_1 \rvert < \lvert \Delta_3 \rvert$ analogously. We may assume that
\[ x_1^a + x_2^a +x_3^a = A\]
with $x_3$ dominant. We may find a conjugate $(x_1', x_2', x_3')$ of $(x_1, x_2, x_3)$ such that $x_3'$ is not dominant. We then verify in PARI that
\[\Big\lvert \Big(\frac{x_1}{x_3}\Big) \Big\rvert + \Big\lvert \Big(\frac{x_1'}{x_3}\Big) \Big\rvert +\Big\lvert \Big(\frac{x_2}{x_3}\Big) \Big\rvert +\Big\lvert \Big(\frac{x_2'}{x_3}\Big) \Big\rvert +\Big\lvert \Big(\frac{x_3'}{x_3}\Big) \Big\rvert <1.\]
This method allows us to eliminate these $9$ triples.

Now suppose that $\Delta_3$ was listed. Once again, we may find in PARI all the possible triples $(\Delta_1, \Delta_2, \Delta_3)$. We then eliminate each triple in the above way, taking either $x_1$ or $x_3$ to be dominant, according as to whether $\lvert \Delta_1 \rvert > \lvert \Delta_3 \rvert$ or $\lvert \Delta_1 \rvert < \lvert \Delta_3 \rvert$. The one triple in this case which we cannot immediately eliminate in this way is $(-240, -240, -235)$. However, we saw already how to eliminate this triple. The proof of this case is thus complete.

So we reduce to the second case, where $h_1 \geq 128$. Observe that either $\lvert \Delta_1 \rvert > \lvert \Delta_3 \rvert$ or $\lvert \Delta_1 \rvert < \lvert \Delta_3 \rvert$. Let $i \in \{1,3\}$ be such that $\Delta_i$ is the discriminant with strictly larger absolute value. Write $\Delta= \lvert \Delta_i \rvert$. Taking $x_i$ to be dominant, we obtain that
\begin{align}\label{eq:ineq17}
\lvert A \rvert \geq (e^{\pi  \Delta^{1/2}} -2079)^a - (e^{\pi  \Delta^{1/2}/2} + 2079)^a - (e^{\pi  (\Delta - 1)^{1/2}} + 2079)^a.
\end{align}
Since $h_1, h_2, h_3 \geq 64$, we may find a conjugate $(x_1', x_2', x_3')$ with none of $x_1', x_2', x_3'$ dominant, so that
\begin{align} \label{eq:ineq18}
\lvert A \rvert \leq 3 (e^{\pi  \Delta^{1/2}/2} +2079)^a.
\end{align}

The bounds \eqref{eq:ineq17} and \eqref{eq:ineq18} are incompatible when
\[ (e^{\pi  \Delta^{1/2}} -2079)^a - (e^{\pi  (\Delta - 1)^{1/2}} + 2079)^a > 4 (e^{\pi  \Delta^{1/2}/2} +2079)^a,\] 
and, in particular, whenever (without loss of generality, we take $\Delta \geq 6$, so that $e^{\pi \Delta^{1/2}}-2079>0$)
\[ 1 - \Big(\frac{e^{\pi  (\Delta - 1)^{1/2}} + 2079}{e^{\pi  \Delta^{1/2}} -2079}\Big)^a > 4 \Big(\frac{e^{\pi  \Delta^{1/2}/2} +2079}{e^{\pi  \Delta^{1/2}} -2079}\Big)^a.\]
For $\Delta \geq 9$, we have that
\[0 < \frac{e^{\pi  (\Delta - 1)^{1/2}} + 2079}{e^{\pi  \Delta^{1/2}} -2079}, \frac{e^{\pi  \Delta^{1/2}/2} +2079}{e^{\pi  \Delta^{1/2}} -2079} <1. \]
Therefore, for the given bounds to be incompatible, it will suffice to have $\Delta \geq 9$ such that
\[ 1 - \frac{e^{\pi  (\Delta - 1)^{1/2}} + 2079}{e^{\pi  \Delta^{1/2}} -2079}> 4 \frac{e^{\pi  \Delta^{1/2}/2} +2079}{e^{\pi  \Delta^{1/2}} -2079},\]
since $1-y^a \geq 1-y$ and $y^a \leq y$ for $0<y<1$ and $a \geq 1$.
Equivalently, to have $\Delta \geq 9$ such that
\[(e^{\pi  \Delta^{1/2}} -2079) - (e^{\pi  (\Delta - 1)^{1/2}} + 2079) > 4 (e^{\pi  \Delta^{1/2}/2} +2079),\]
which happens for all $\Delta \geq 12$. Since $\Delta < 12$ would imply that $x_i \in \Q$, we are done.

\subsection{The case where $h_1= 2 h_2= 2 h_3$}
Finally, we come to the case that $h_1 = 2 h_2 = 2 h_3$. In particular, $\Delta_1 \neq \Delta_2, \Delta_3$. Thus, by Theorem~\ref{thm:fieldsumpower},
\[ \Q(x_2) = \Q(x_2^a)= \Q(x_1^a + x_3^a) = \Q(x_1, x_3) \supset \Q(x_1).\]
This though contradicts the fact that $h_1 = 2 h_2 > h_2$. The proof of the $\epsilon_1=\epsilon_2=\epsilon_3$ case of Theorem~\ref{thm:sum} is thus complete.

\section{The general case of Theorem~\ref{thm:sum}}\label{sec:diff}

In this section, we prove the ``only if'' direction of Theorem~\ref{thm:sum} in general. Let $\epsilon_1, \epsilon_2, \epsilon_3 \in \{\pm 1\}$ and $a \in \Z_{>0}$. Suppose that $x_1, x_2, x_3$ are singular moduli such that
\[ \epsilon_1 x_1^a + \epsilon_2 x_2^a + \epsilon_3 x_3^a \in \Q.\]
Write $\Delta_i$ for the discriminant of $x_i$ and $h_i = h(\Delta_i)$ for the corresponding class number. We will show that we must be in one of cases (1)--(4) of Theorem~\ref{thm:sum}.

If $x_1 = x_2 = x_3$, then clearly $x_1, x_2, x_3 \in \Q$ by Proposition~\ref{prop:fieldpower}.

If $x_i = x_j \neq x_k$, then either $\epsilon_i = - \epsilon_j$ and so $x_k \in \Q$ by Proposition~\ref{prop:fieldpower}, or $\epsilon_i = \epsilon_j$ and so $2 \epsilon_i x_i^a + \epsilon_k x_k^a \in \Q$. The former possibility is case (2) of Theorem~\ref{thm:sum}. So suppose we are in the latter case. Then, by Theorem~\ref{thm:sum2}, either $x_i, x_k \in \Q$ (and we are in case (1) of Theorem~\ref{thm:sum}), or $\Q(x_i) = \Q(x_k)$ and this number field has degree $2$ over $\Q$. 

Suppose that we are in the second case. If $\epsilon_i = \epsilon_k$, then we showed how to eliminate this case at the beginning of Section~\ref{sec:sum}. So we may assume that $\epsilon_i \neq \epsilon_k$. Let $A \in \Q$ be such that
\[2 x_i^a - x_k^a = A.\]
Since $\Q(x_i)=\Q(x_k)$ has degree $2$, there exists a unique conjugate $(x_i', x_k')$ of $(x_i, x_k)$. One has that
\[2 (x_i')^a - (x_k')^a = A.\]

Suppose that $\Delta_i = \Delta_k$. Then $x_i' = x_k$ and $x_k' = x_i$. One thus obtains that
\[2x_i^a - x_k^a = 2 x_k^a -x_i^a,\]
and so $x_i^a = x_k^a$. This contradicts Theorem~\ref{thm:prod2}. So we may suppose that $\Delta_i \neq \Delta_k$. One then uses the equality
\[ 2 x_i^a - x_k^a = 2 (x_i')^a - (x_k')^a\]
to eliminate all the possibilities using PARI, in the same way as explained at the beginning of Section~\ref{sec:sum}.

So we may assume from now on that $x_1, x_2, x_3$ are pairwise distinct. Suppose next that some $x_i \in \Q$. Then $\epsilon_j x_j^a + \epsilon_k x_k^a \in \Q$. By Theorem~\ref{thm:fieldsumpower},
\[\Q = \Q(\epsilon_j x_j^a + \epsilon_k x_k^a) = \Q(x_j, x_k)\]
and so $x_j, x_k \in \Q$, unless $\epsilon_j = \epsilon_k$ and $\Delta_j = \Delta_k$. In the exceptional case, $[\Q(x_j, x_k) : \Q] \leq 2$, and so $x_j, x_k$ are distinct, conjugate singular moduli of degree $2$. Since $\epsilon_j = \epsilon_k$, this falls under case (3) of Theorem~\ref{thm:sum}. So we also assume subsequently that $x_1, x_2, x_3 \notin \Q$. 

If $\epsilon_1 = \epsilon_2 = \epsilon_3$, then we are done by the $\epsilon_1 = \epsilon_2 = \epsilon_3$ case of Theorem~\ref{thm:sum}, which was proved in Section~\ref{sec:sum}. So we may assume that the $\epsilon_i$ are not all equal. Clearly, it suffices to consider the case where $\epsilon_1 = \epsilon_2 = - \epsilon_3 = 1$. Let $A \in \Q$ be such that
\[ x_1^a + x_2^a - x_3^a = A.\]
By Theorem~\ref{thm:fieldsumpower}, we have that
\[ \Q(x_1) = \Q(x_1^a)= \Q(x_2^a - x_3^a) = \Q(x_2, x_3)\]
and 
\[\Q(x_2) = \Q(x_2^a)=\Q(x_1^a - x_3^a) = \Q(x_1, x_3).\]
So
\[ \Q(x_1, x_2)= \Q(x_1) = \Q(x_2) \supset \Q(x_3).\]
Further, $\Q(x_3) =\Q(x_3^a) = \Q(x_1^a + x_2^a)$ and $[\Q(x_1, x_2) : \Q(x_1^a + x_2^a)] \leq 2$.

\subsection{The case where $\Q(x_1) = \Q(x_2) = \Q(x_3)$}\label{subsec:diffeq}

Suppose first that $\Q(x_1) = \Q(x_2) = \Q(x_3)$. Write $h = h(\Delta_1) = h(\Delta_2) = h(\Delta_3)$. Without loss of generality, we may assume that $\lvert \Delta_1 \rvert \geq \lvert \Delta_2 \rvert$. If $\Q(\sqrt{\Delta_i}) \neq \Q(\sqrt{\Delta_j})$ for some $i \neq j$, then, by Lemma~\ref{lem:samefield}, all possible $(\Delta_1, \Delta_2, \Delta_3)$ are listed. 

We now explain how we eliminate these possibilities. To begin with, assume that $h = 2$. Then there is a unique non-trivial conjugate $(x_1', x_2', x_3')$ of $(x_1, x_2, x_3)$ and one has that
\[(x_1')^a + (x_2')^a - (x_3')^a = A.\]
If $\Delta_1 = \Delta_2$, then $x_1, x_2$ are distinct, conjugate singular moduli of degree $2$, so one has that $x_1 = x_2'$ and $x_2 = x_1'$. Thus, in this case, one obtains that $(x_3)^a - (x_3')^a=0$, but this contradicts Theorem~\ref{thm:prod2} since $x_3 \neq x_3'$. Hence, $\Delta_1 \neq \Delta_2$, and so $\lvert \Delta_1 \rvert > \lvert \Delta_2 \rvert$.

Suppose first that $\Delta_1 = \Delta_3$. Then $x_1' = x_3$ and $x_3' = x_1$. So
\[ 2 x_1^a = 2 x_3^a - x_2^a + (x_2')^a.\]
We may then use this equality, as at the beginning of Section~\ref{sec:sum}, to eliminate this choice of discriminants by taking $x_1$ to be dominant.

Suppose then that $\Delta_1 \neq \Delta_3$. Then 
\begin{align}\label{eq:2}
x_1^a = (x_1')^a - x_2^a + (x_2')^a + x_3^a - (x_3')^a
\end{align}
and
\begin{align}\label{eq:3}
x_3^a = x_1^a -(x_1')^a + x_2^a - (x_2')^a + (x_3')^a.
\end{align}
If $\lvert \Delta_1 \rvert > \lvert \Delta_3 \rvert$, then take $x_1$ dominant and use \eqref{eq:2}. If $\lvert \Delta_1 \rvert < \lvert \Delta_3 \rvert$, then take $x_3$ dominant and use \eqref{eq:3}. In each case, the approach is the same as that used in \S\ref{subsubsec1}.

So we may now assume that $h>2$. Then inspection of \cite[Table~2]{AllombertBiluMadariaga15} gives us that $h \geq 4$. Recall that, by assumption, $\lvert \Delta_1 \rvert \geq \lvert \Delta_2 \rvert$. 

Assume first that $\lvert \Delta_1 \rvert \geq \lvert \Delta_3 \rvert$ as well. Suppose that $x_1$ is dominant. Since $h \geq 4$, there exists a conjugate $(x_1', x_2', x_3')$ of $(x_1, x_2, x_3)$ such that no $x_i'$ is dominant. Proceeding as above, we have that 
\[ 1 = \Big(\frac{x_1'}{x_1}\Big)^a -\Big(\frac{x_2}{x_1}\Big)^a +\Big(\frac{x_2'}{x_1}\Big)^a +\Big(\frac{x_3}{x_1}\Big)^a -\Big(\frac{x_3'}{x_1}\Big)^a.\]
Further, $\lvert x_1' \rvert, \lvert x_2 \rvert, \lvert x_2' \rvert, \lvert x_3 \rvert, \lvert x_3' \rvert < \lvert x_1 \rvert$ by Proposition~\ref{prop:dominc}, since: $\lvert \Delta_1 \rvert \geq \lvert \Delta_2 \rvert, \lvert \Delta_3 \rvert$; no $x_i'$ is dominant; and, for $i \in \{2,3\}$, if $\Delta_i = \Delta_1$, then $x_i$ is not dominant since $x_i \neq x_1$. Consequently,
\begin{align*}
1=	&	\Big\lvert \Big(\frac{x_1'}{x_1}\Big)^a -\Big(\frac{x_2}{x_1}\Big)^a +\Big(\frac{x_2'}{x_1}\Big)^a +\Big(\frac{x_3}{x_1}\Big)^a -\Big(\frac{x_3'}{x_1}\Big)^a \Big\rvert\\
	\leq &\Big\lvert \Big(\frac{x_1'}{x_1}\Big) \Big\rvert + \Big\lvert \Big(\frac{x_2}{x_1}\Big) \Big\rvert +\Big\lvert \Big(\frac{x_2'}{x_1}\Big) \Big\rvert +\Big\lvert \Big(\frac{x_3}{x_1}\Big) \Big\rvert +\Big\lvert \Big(\frac{x_3'}{x_1}\Big) \Big\rvert.
\end{align*}
Once again, we may verify in PARI that this last expression is $<1$ for every possible choice of $x_2, x_3$ for each of the relevant triples $(\Delta_1, \Delta_2, \Delta_3)$. We may thereby eliminate this case.

Now suppose that $\lvert \Delta_3 \rvert > \lvert \Delta_1 \rvert$. Then we proceed in exactly the same way, but taking $x_3$ dominant instead. We may thus eliminate each of the relevant triples $(\Delta_1, \Delta_2, \Delta_3)$.

So suppose that $\Q(\sqrt{\Delta_1}) = \Q(\sqrt{\Delta_2}) = \Q(\sqrt{\Delta_3})$. Then, by Lemma~\ref{lem:samefield} again, $\Delta_i / \Delta_j \in \{1/4, 1, 4\}$ for all $i, j$. We thus have, since $\lvert \Delta_1 \rvert \geq \lvert \Delta_2 \rvert$ by assumption, one of:
\begin{enumerate}
	\item either $\Delta_1 = \Delta_2$ and one of:
	\begin{enumerate}
		\item $\Delta_3 = \Delta_1 / 4$,
		\item $\Delta_3 = \Delta_1$,
		\item $\Delta_3 = 4 \Delta_1$;
	\end{enumerate}
	\item or $\Delta_1 = 4 \Delta_2$ and one of:
	\begin{enumerate}
		\item $\Delta_3 = \Delta_1$,
		\item $\Delta_3 = \Delta_2$.
	\end{enumerate}
\end{enumerate}

Cases 1(a) and 2(a) reduce to the second case of \S\ref{subsubsec1} by means of the inequalities
\[ \lvert x_1^a + x_2^a - x_3^a \rvert \geq \lvert x_1 \rvert^a - \lvert x_2 \rvert^a - \lvert x_3 \rvert^a\]
and
\[\lvert x_1^a + x_2^a - x_3^a \rvert \leq \lvert x_1 \rvert^a + \lvert x_2 \rvert^a + \lvert x_3 \rvert^a.\]

In case 1(b), we must have that $h \geq 3$ since $x_1, x_2, x_3$ are three distinct Galois conjugates. If $h=3$, then $x_1, x_2, x_3$ are the three roots of the Hilbert class polynomial $H_{\Delta_1} \in \Z[x]$. Hence,
\[ x_1^a + x_2^a + x_3^a \in \Q.\]
Since $x_1^a + x_2^a - x_3^a \in \Q$ also, we obtain that
\[ x_3^a = \frac{(x_1^a + x_2^a + x_3^a)-(x_1^a+x_2^a-x_3^a)}{2} \in \Q.\]
But then Proposition~\ref{prop:fieldpower} implies that $x_3 \in \Q$, a contradiction. So we may assume that $h \geq 4$. We may now reduce case 1(b) to the case of \S\ref{subsubsec0} by using the inequalities from the previous paragraph.

Case 1(c) is new. Here, we may take $x_3$ to be dominant and use the lower bound coming from
\[ \lvert x_1^a + x_2^a - x_3^a \rvert \geq \lvert x_3 \rvert^a - \lvert x_1 \rvert^a - \lvert x_2 \rvert^a.\]
The resulting lower bound is incompatible 
for all $\lvert \Delta_1 \rvert \geq 3$ (i.e. for all $\Delta_1$) with the upper bound obtained from
\[\lvert (x_1')^a + (x_2')^a - (x_3')^a \rvert \leq \lvert x_1' \rvert^a + \lvert x_2' \rvert^a + \lvert x_3' \rvert^a,\]
where $(x_1', x_2', x_3')$ is a conjugate of $(x_1, x_2, x_3)$ with $x_3'$ not dominant. Such a conjugate exists since $h \geq 2$.

We treat case 2(b) in directly analogous fashion to case 1(c), swapping the roles of $x_1$ and $x_3$. We thus complete the proof in this case.

\subsection{The case where $\Q(x_1) = \Q(x_2) \supsetneq \Q(x_3)$}\label{subsec:diffdiff}

So now suppose that $\Q(x_1) = \Q(x_2) \supsetneq \Q(x_3)$. Then $[\Q(x_1, x_2) : \Q(x_3)]=2$. If $\Delta_1 \neq \Delta_2$, then 
\[\Q(x_3) = \Q(x_3^a) = \Q(x_1^a + x_2^a) = \Q(x_1, x_2) = \Q(x_1),\]
a contradiction. So $\Delta_1 = \Delta_2$. Observe also that $h_1 = h_2 \geq 4$ since $h_3 \geq 2$ and $h_1 = h_2 = 2 h_3$. Suppose first that $\Q(\sqrt{\Delta_1})=\Q(\sqrt{\Delta_3})$. Then Lemma~\ref{lem:subfieldsamefund} implies that $\Delta_1 \in \{9 \Delta_3/4, 4 \Delta_3, 9 \Delta_3, 16 \Delta_3\}$. These cases may be dealt with in exactly the same way as the corresponding cases in Section~\ref{sec:sum}.

So suppose that $\Q(\sqrt{\Delta_1}) \neq \Q(\sqrt{\Delta_3})$. Then Lemma~\ref{lem:subfielddiffund} implies that either one of $\Delta_1, \Delta_3$ is listed, or $h_1 \geq 128$. In the first case, we may find all possibilities for $(\Delta_1, \Delta_2, \Delta_3)$. Each of these may be eliminated in exactly the same way as was done in Section~\ref{sec:sum}. So we reduce to the second case, where $h_1 \geq 128$. This case may also be handled exactly as was done in Section~\ref{sec:sum}. The proof of Theorem~\ref{thm:sum} is thus complete.

\section{Fermat's last theorem for singular moduli}\label{sec:fermat}

In this section, we prove that there are no triples $(x_1, x_2, x_3)$ with $x_1, x_2, x_3$ all non-zero singular moduli which satisfy either
\[ x_1^a + x_2^a + x_3^a = 0\]
or
\[ x_1^a + x_2^a - x_3^a = 0,\]
where $a \in \Z_{>0}$. This proves a ``Fermat's last theorem'' for singular moduli and gives a completely explicit form of Andr\'e--Oort for the two corresponding algebraic surfaces. For fixed $a$, we note that effective bounds on the discriminants of singular moduli $x_1, x_2, x_3$ satisfying either of the above equations follow from Binyamini's effective Andr\'e--Oort result for ``hdnd hypersurfaces'' \cite[Corollary~4]{Binyamini19}.

\begin{cor}\label{cor:diff}
	Let $a \in \Z_{>0}$. Suppose that $x_1, x_2, x_3$ are singular moduli such that $x_1^a + x_2^a - x_3^a = 0$. Then $\{0, x_3\}=\{x_1, x_2\}$.
\end{cor}

\begin{proof}
	Suppose that $x_1, x_2, x_3$ are singular moduli such that $x_1^a + x_2^a - x_3^a = 0$. If $x_3 \in \{x_1, x_2\}$, then clearly $\{x_1, x_2\} = \{0, x_3\}$. Suppose next that $x_1 = x_2 \neq x_3$. Then $2 x_1^a = x_3^a$. So $x_1, x_3 \neq 0$ and
	\[ \frac{x_1^a}{x_3^a} = \frac{1}{2}.\]
	We thus must have that $x_1, x_3 \in \Q$ by Theorem~\ref{thm:prod2}. An inspection of the list of rational singular moduli shows that this is impossible.
	
	We may thus assume that $x_1, x_2, x_3$ are pairwise distinct. We will show that this impossible. Suppose first that some $x_i = 0$. Then $x_j, x_k \neq 0$ and $x_j^a + \epsilon x_k^a =0$ for some $\epsilon \in \{\pm 1\}$. Hence one has that
	\[ \Big(\frac{x_j}{x_k}\Big)^{2a} = 1,\]
	and so, by Theorem~\ref{thm:prod2}, one has that $x_j, x_k \in \Q$. So $x_j/ x_k$ is a rational root of unity (i.e. $\pm 1$), so $x_j = - x_k$ since $x_j \neq x_k$. Inspecting the list of rational singular moduli, we see that this is impossible.
	
	Hence, we may assume that $x_1, x_2, x_3 \neq 0$ also. We will show that this leads to a contradiction. By Theorem~\ref{thm:sum}, we have that either $x_1, x_2, x_3 \in \Q$ or $x_3 \in \Q$ and $x_1, x_2$ are degree $2$ and conjugate. Suppose first that $x_1, x_2, x_3 \in \Q$ (and hence $\in \Z$). Then Fermat's last theorem (Wiles \cite{Wiles95} and Taylor--Wiles \cite{TaylorWiles95}) implies that $a \leq 2$. For $a \leq 2$, we may verify in PARI that the statement holds.
	
	Suppose then that $x_3 \in \Q$ and $x_1, x_2$ are degree $2$ and conjugate. Without loss of generality, we may assume that $x_1$ is dominant, and so $x_2$ is not dominant. So $\lvert x_2 \rvert \leq 0.1 \lvert x_1 \rvert$ by Proposition~\ref{prop:dom}. We thus have that
	\begin{align*} 
		\lvert x_3 \rvert^a &\geq \lvert x_1 \rvert^a - \lvert x_2 \rvert^a\\
		&\geq \lvert x_1 \rvert^a - (0.1 \lvert x_1 \rvert)^a\\
		&= (1- 0.1^a) \lvert x_1 \rvert^a\\
		&\geq (0.9 \lvert x_1 \rvert)^a.
	\end{align*}
Hence, $\lvert x_3 \rvert \geq  0.9 \lvert x_1 \rvert$. Clearly,
\[ \lvert x_3 \rvert^a \leq \lvert x_1 \rvert^a + \lvert x_2 \rvert^a \leq (\lvert x_1 \rvert + \lvert x_2 \rvert)^a,\]
and so $\lvert x_3 \rvert \leq \lvert x_1 \rvert + \lvert x_2 \rvert$. We may verify in PARI that these two inequalities are never both satisfied.	
	\end{proof}

\begin{remark}
	In particular, for every $a \in \Z_{>0}$, the sum (respectively, difference) of the $a$th powers of two non-zero (respectively, non-zero and distinct) singular moduli is never equal to the $a$th power of a singular modulus. Differences of singular moduli are themselves objects of considerable interest, see e.g. \cite{GrossZagier85}. 
\end{remark}

\begin{cor}\label{cor:vanish}
	Let $n \leq 3$ and let $x_1, \ldots, x_n$ be singular moduli (which are not necessarily pairwise distinct). Suppose that $x_1^a + \dots + x_n^a = 0$ for some $a \in \Z_{>0}$. Then $x_i=0$ for all $i \in \{1, \ldots, n\}$.
\end{cor}

\begin{proof}
	The case $n=1$ is trivial. Suppose $n=2$. If $x_1 = x_2$, then the result is immediate. So suppose $x_1 \neq x_2$ and $x_1^a + x_2^a =0$. The same argument as in the case where some $x_i=0$ in the proof of Corollary~\ref{cor:diff} shows that this is impossible.
	
	Now for the $n=3$ case. Suppose that $x_1, x_2, x_3$ are singular moduli such that $x_1^a + x_2^a + x_3^a = 0$ for some $a \geq 1$. If $x_1 = x_2 = x_3$, then clearly $x_1=x_2=x_3=0$. Suppose that $x_1=x_2 \neq x_3$. Then $2 x_1^a = -x_3^a$. So $x_1, x_3 \neq 0$. Then Theorem~\ref{thm:prod2} implies that $x_1, x_3 \in \Q$. An inspection of the list of rational singular moduli shows that this is impossible.
	
	So we may assume that $x_1, x_2, x_3$ are pairwise distinct. Suppose $x_3=0$ say. Then $x_1^a + x_2^a = 0$. By the same argument as before, this is impossible. So we may assume also that $x_1, x_2, x_3 \neq 0$. Theorem~\ref{thm:sum} implies that one of the following holds:
	\begin{enumerate}
		\item $x_1, x_2, x_3 \in \Q$,
		\item some $x_i, x_j$ are conjugate and of degree $2$ and the remaining $x_k \in \Q$,
		\item $x_1, x_2, x_3$ are conjugate and of degree $3$.
	\end{enumerate}

Suppose first that $x_1, x_2, x_3 \in \Q$. Then Fermat's last theorem \cite{Wiles95, TaylorWiles95} implies that $a \leq 2$. If $a =2$, then $- x_3^2 = x_1^2 + x_2^2 >0$, which is impossible. So we must have that $a=1$. But $a=1$ may be eliminated by a computation.

Suppose next that $x_1, x_2$ are conjugate and of degree $2$ and $x_3 \in \Q$. This case may be dealt with as in Corollary~\ref{cor:diff}. We assume that $x_1$ is dominant and use the inequalities
\[0.9 \lvert x_1 \rvert \leq \lvert x_3 \rvert \leq \lvert x_1 \rvert + \lvert x_2 \rvert. \]

Suppose finally that $x_1, x_2, x_3$ are conjugate and of degree $3$. We may assume that $x_1$ is dominant, so $x_2, x_3$ are not dominant. Then
\begin{align*} 
	0=\lvert x_1^a + x_2^a + x_3^a \rvert &\geq \lvert x_1 \rvert^a - \lvert x_2 \rvert^a - \lvert x_3 \rvert^a\\
	&\geq \lvert x_1 \rvert^a -  (0.1 \lvert x_1 \rvert)^a - (0.1 \lvert x_1 \rvert)^a\\
	&\geq (1- 2 \times 0.1^a) \lvert x_1 \rvert^a\\
	&>0,
	\end{align*}
which is a contradiction. The proof is thus complete.
	\end{proof}

Corollary~\ref{cor:sum} follows immediately from Corollaries~\ref{cor:diff} and \ref{cor:vanish}.

\section{Powers of a product of three singular moduli}\label{sec:prod}

Now we come to the proof of Theorem~\ref{thm:prod}. Again the ``if'' direction is immediate, so we just need to prove the ``only if''. Suppose that $x_1, x_2, x_3$ are singular moduli of respective discriminants $\Delta_1, \Delta_2, \Delta_3$ such that
\[(x_1 x_2 x_3)^a = A\]
for some $a \in \Z \setminus \{0\}$ and $A \in \Q^{\times}$. Write $h_i$ for the respective class numbers $h(\Delta_i)$. Clearly, we may assume that $a \geq 1$. We will show that one of cases (1)--(3) of Theorem~\ref{thm:prod} must hold.

\subsection{The trivial cases}\label{subsec:trivial}

If $x_1 = x_2 = x_3$, then we must have that $x_1, x_2, x_3 \in \Q^{\times}$ by Proposition~\ref{prop:fieldpower}. If the set $\{x_1, x_2, x_3\}$ has cardinality $2$, then we are done by Theorem~\ref{thm:prod2}. So we may and do assume that $x_1, x_2, x_3$ are pairwise distinct. 

If some $x_i \in \Q$, then again the result follows from Theorem~\ref{thm:prod2}. So we assume from now on that $x_1, x_2, x_3 \notin \Q$. In particular, $h_1, h_2, h_3 \geq 2$. If $\Delta_1 = \Delta_2 = \Delta_3$, then we must have that $h_1 = h_2 = h_3 \geq 3$, since $x_1, x_2, x_3$ are all distinct. If $\Delta_1 = \Delta_2 = \Delta_3$ and $h_1 = h_2 = h_3 = 3$, then we are in case (3) of Theorem~\ref{thm:prod}. So from now on, assume that we do not have both $\Delta_1 = \Delta_2 = \Delta_3$ and $h_1 = h_2 = h_3 = 3$.

\subsection{Bounding the non-trivial cases effectively}\label{subsec:bdnontriv}
We are not then in any of cases (1)--(3) of Theorem~\ref{thm:prod}. We will show that this cannot happen. Our first step is to reduce all the possibilities for $(\Delta_1, \Delta_2, \Delta_3)$ to a finite list. This we do in this subsection. In the next subsection, we will explain how we can eliminate each of the cases in this list in PARI.

Without loss of generality, we may assume that $h_1 \geq h_2 \geq h_3 \geq 2$. By Proposition~\ref{prop:fieldpower}, we have that 
\[ \Q(x_1) = \Q(x_1^{-a}) = \Q(x_2^a x_3^a).\]
Hence, by Theorem~\ref{thm:fieldprodpower}, either $\Q(x_1) = \Q(x_2, x_3)$ or $[\Q(x_2, x_3) : \Q(x_1)] = 2$. Obviously, we obtain the same swapping the roles of $x_1, x_2, x_3$. Since $[\Q(x_i) : \Q] = h_i$, we may thus conclude that one of the following holds:
\begin{enumerate}
	\item $h_1 = h_2 = h_3$,
	\item $h_1 = h_2 = 2 h_3$,
	\item $h_1 = 2 h_2 = 2 h_3$.
\end{enumerate} 

We may now proceed to bound $\lvert \Delta_1 \rvert, \lvert \Delta_2 \rvert, \lvert \Delta_3 \rvert$ exactly as in \cite{Fowler20}. In particular, we obtain the same bounds as there thanks to the elementary fact that, for $s, t\geq 0$, one has that $s^a > t^a$ if and only if $s > t$. We thus reduce to the same possibilities for $(\Delta_1, \Delta_2, \Delta_3)$ as in \cite{Fowler20}, namely those recorded in the list at the beginning of \S4 of \cite{Fowler20}. 

\subsection{Eliminating the non-trivial cases}\label{subsec:elimnontriv}

For a triple $(\Delta_1, \Delta_2, \Delta_3)$ belonging to this list, we may eliminate it using a PARI script in the following way. 

Suppose that $x_1, x_2, x_3$ are pairwise distinct, non-zero singular moduli of discriminants $\Delta_1, \Delta_2, \Delta_3$ such that $(x_1 x_2 x_3)^a = A$ for some $a \in \Z_{>0}$ and $A \in \Q^\times$. Let $L$ be a Galois extension of $\Q$ containing $x_1, x_2, x_3$. Then, for every $\sigma \in \Gal(L / \Q)$, we have that
\[(\sigma(x_1) \sigma(x_2) \sigma(x_3))^a = A,\]
and thus 
\[ \frac{x_1 x_2 x_3}{\sigma(x_1) \sigma(x_2) \sigma(x_3)}\]
is a root of unity in $L$.

Therefore, if we can find an automorphism $\sigma \in \Gal(L / \Q)$ such that 
\[ \frac{x_1 x_2 x_3}{\sigma(x_1) \sigma(x_2) \sigma(x_3)}\]
is not among the roots of unity in $L$ for each possible choice of $x_1, x_2, x_3$, then we may eliminate the triple $(\Delta_1, \Delta_2, \Delta_3)$. Using PARI, we eliminate in this way all the triples $(\Delta_1, \Delta_2, \Delta_3)$ belonging to the list from Subsection~\ref{subsec:elimnontriv}, and thereby complete the proof of Theorem~\ref{thm:prod}.

\begin{remark}\label{rmk:multfmt}
	The obvious multiplicative analogue of the results in Section~\ref{sec:fermat} would concern solutions in singular moduli to the equation $(x_1 x_2 x_3)^a =1$, where $a \in \Z \setminus \{0\}$. One may easily show, using Theorems~\ref{thm:prod} and \ref{thm:prod2} together with some easy computations, that there are no such solutions. However, the non-existence of such solutions is also an immediate consequence of Bilu, Habegger, and K\"uhne's result \cite{BiluHabeggerKuhne18} that no singular modulus is an algebraic unit. Indeed, the result of \cite{BiluHabeggerKuhne18} implies that there are no solutions in singular moduli to the equation $(\prod_{i=1}^n x_i)^a = 1$ for any $n \in \Z_{>0}$ and $a \in \Z \setminus \{0\}$.
\end{remark}

\section{The proof of Theorem~\ref{thm:quot}}\label{sec:quot}

Now suppose that $x_1, x_2, x_3$ are singular moduli of respective discriminants $\Delta_1, \Delta_2, \Delta_3$ such that
\[(x_1^{\epsilon_1} x_2^{\epsilon_2} x_3^{\epsilon_3})^a = A\] 
for some $a \in \Z \setminus \{0\}$, $A \in \Q^\times$, and $\epsilon_i \in \{\pm 1\}$. As usual, write $h_i = h(\Delta_i)$. Clearly, we may assume that $a \geq 1$. If $\epsilon_1 = \epsilon_2 = \epsilon_3$, then the desired result is just Theorem~\ref{thm:prod}. So we may assume that the $\epsilon_i$ are not all equal.

If the $x_i$ are not pairwise distinct, then the result follows easily from Proposition~\ref{prop:fieldpower} and Theorem~\ref{thm:prod2}. So we may suppose that $x_1, x_2, x_3$ are pairwise distinct. In addition, if some $x_i \in \Q$, then the result also follows immediately from Theorem~\ref{thm:prod2}. So we assume also that no $x_i$ is rational. In particular, $h_1, h_2, h_3 \geq 2$.  We are thus not in any of the trivial cases (1)--(4) of Theorem~\ref{thm:quot}. We will show this leads to a contradiction, apart from in the one exceptional case (5).

Without loss of generality, we may assume that $\epsilon_1 = \epsilon_2 = - \epsilon_3 = 1$, so that
\[\Big(\frac{x_1 x_2}{x_3}\Big)^a = A.\]
By Proposition~\ref{prop:fieldpower} and Theorem~\ref{thm:fieldprodpower}, we have that
\[ \Q(x_1) = \Q(x_1^a) = \Q\Big(\frac{x_2^a}{x_3^a}\Big) = \Q(x_2, x_3)\]
and 
\[ \Q(x_2) =  \Q(x_2^a) = \Q\Big(\frac{x_1^a}{x_3^a}\Big) = \Q(x_1, x_3).\]
Thus, $\Q(x_1) = \Q(x_2) \supset \Q(x_3)$.
Also, the field $\Q(x_3) =\Q(x_3^a) = \Q(x_1^a x_2^a)$ is equal to the field $\Q(x_1, x_2)$ if $\Delta_1 \neq \Delta_2$ and satisfies $[ \Q(x_1, x_2) : \Q(x_3)] \leq 2$ if $\Delta_1 = \Delta_2$. We now distinguish into the two cases: $\Q(x_1) = \Q(x_2) = \Q(x_3)$ and $\Q(x_1) = \Q(x_2) \supsetneq \Q(x_3)$. 

In our proof, we repeatedly use the following fact. Let $L$ be a Galois extension containing $x_1, x_2, x_3$. Then, for any $\sigma \in \Gal(L/\Q)$, we have that
\[ \Big ( \frac{\sigma(x_1) \sigma(x_2)}{\sigma(x_3)} \Big )^a = A\]
and hence
\[ \Big\lvert \frac{\sigma(x_1) \sigma(x_2)}{\sigma(x_3)} \Big\rvert = \Big\lvert \frac{x_1 x_2}{x_3} \Big\rvert = \lvert A \rvert^{1/a}.\]

In Subsections~\ref{subsecA} and \ref{subsecB}, we will show that either we are in case (5) of Theorem~\ref{thm:quot}, or the triple $(\Delta_1, \Delta_2, \Delta_3)$ belongs to a finite list of possible exceptions, which we may find effectively. In Subsection~\ref{subsecC}, we will then show that in fact none of these finitely many possible exceptions can occur. This will complete the proof of Theorem~\ref{thm:quot}.

\subsection{The case where $\Q(x_1) = \Q(x_2) = \Q(x_3)$.}\label{subsecA}

\subsubsection{The subcase where $\Delta_1 = \Delta_2 = \Delta_3$.}

Write $\Delta = \Delta_1 =  \Delta_2 = \Delta_3$ and $h = h_1 = h_2 = h_3$. Note that $h \geq 3$ since the $x_i$ are all distinct. Taking conjugates, we assume that $x_1$ is dominant, and thus obtain, using \eqref{eq:ineq1} and Lemma~\ref{lem:lower}, the lower bound
\[\lvert A \rvert^{1/a} \geq \frac{(e^{\pi \lvert \Delta \rvert^{1/2}}-2079)\min \{4.4 \times 10^{-5}, 3500 \lvert \Delta \rvert^{-3}\}}{e^{\pi \lvert \Delta \rvert^{1/2}/2}+2079}.\]
Conjugating again, we assume that $x_3$ is dominant. We obtain the upper bound
\[\lvert A \rvert^{1/a} \leq \frac{(e^{\pi \lvert \Delta \rvert^{1/2}/2}+2079)^2}{e^{\pi \lvert \Delta \rvert^{1/2}}-2079}.\]
These two bounds are incompatible when $\lvert \Delta \rvert \geq 43$.

\subsubsection{The subcase where $\Q(\sqrt{\Delta_i}) \neq \Q(\sqrt{\Delta_j})$ for some $i, j$.}
Then Lemma~\ref{lem:samefield} implies that $\Delta_1, \Delta_2, \Delta_3$ are all listed in \cite[Table~2]{AllombertBiluMadariaga15}.

\subsubsection{The subcase where $\Delta_1, \Delta_2, \Delta_3$ are not all equal, but $\Q(\sqrt{\Delta_1})= \Q(\sqrt{\Delta_2})= \Q(\sqrt{\Delta_3})$.}\label{ss}

Then Lemma~\ref{lem:samefield} implies that $\Delta_i / \Delta_j \in \{1/4, 1, 4\}$ for all $i, j$. Without loss of generality, assume that $\lvert \Delta_1 \rvert \geq \lvert \Delta_2 \rvert$. Then one of the following must hold:
\begin{enumerate}
	\item $\Delta_1 = \Delta_2 = 4 \Delta_3$,
	\item $4 \Delta_1 = 4 \Delta_2 = \Delta_3$,
	\item $\Delta_1 = 4 \Delta_2 = \Delta_3$,
	\item $\Delta_1 = 4 \Delta_2 = 4 \Delta_3$.
\end{enumerate}

First, suppose that $\Delta_1 = \Delta_2 = 4 \Delta_3$. Write $\Delta = \Delta_3$. Then Lemma~\ref{lem:samefield} implies also that $\Delta \equiv 1 \bmod 8$. Therefore, by Lemma~\ref{lem:dom}, there are no subdominant singular moduli of discriminant $4 \Delta$ and there are precisely two subdominant singular moduli of discriminant $\Delta$, provided $\Delta \notin \{-7, -15\}$. If $\Delta = -7$, then $x_3 \in \Q$, which is ruled out by assumption. Suppose that $\Delta = -15$. Then $h(4 \Delta) = h(\Delta) = 2$. So $x_1, x_2$ are the two roots of $H_{4 \Delta} \in \Z[z]$, and hence $x_1 x_2 \in \Q$. Thus $x_3^a \in \Q$, and so $x_3 \in \Q$ by Proposition~\ref{prop:fieldpower}. So we may assume that $\Delta \notin \{-7, -15\}$.

Since $\Q(x_1) = \Q(x_2) = \Q(x_3)$, each conjugate of $x_i$ occurs precisely once as the $i$th coordinate of a conjugate $(x_1', x_2', x_3')$ of $(x_1, x_2, x_3)$. In particular, for each $i$, there is exactly one conjugate $(x_1', x_2', x_3')$ with $x_i'$ dominant. There is therefore a conjugate $(x_1', x_2', x_3')$ with $x_3'$ not dominant and one of $x_1', x_2'$ dominant. This gives the lower bound
\[ \lvert A \rvert^{1/a} \geq \frac{(e^{2 \pi \lvert \Delta \rvert^{1/2}} - 2079)\min \{4.4 \times 10^{-5}, 3500 \times 4^{-3} \lvert \Delta \rvert^{-3}\}}{e^{\pi \lvert \Delta \rvert^{1/2}/2}+2079}.\]

Suppose for now that the conjugate $(x_1'', x_2'', x_3'')$ with $x_3''$ dominant has one of $x_1'', x_2''$ dominant. Since there are two subdominant singular moduli of discriminant $\Delta$, there is then another conjugate $(x_1''', x_2''', x_3''')$ with $x_3'''$ subdominant and neither $x_1''', x_2'''$ dominant. This conjugate $(x_1''', x_2''', x_3''')$ gives the upper bound
\[ \lvert A \rvert^{1/a} \leq \frac{(e^{2 \pi \lvert \Delta \rvert^{1/2}/3} +2079)^2}{e^{\pi \lvert \Delta \rvert^{1/2}/2} - 2079}.\]
These two bounds are incompatible whenever $\lvert \Delta \rvert \geq 29$.

Now suppose that the conjugate $(x_1'', x_2'', x_3'')$ with $x_3''$ dominant has neither of $x_1'', x_2''$ dominant. Then this conjugate gives the upper bound 
\[ \lvert A \rvert^{1/a} \leq \frac{(e^{2 \pi \lvert \Delta \rvert^{1/2}/3} +2079)^2}{e^{\pi \lvert \Delta \rvert^{1/2}} - 2079}.\]
This is incompatible with the previous lower bound if $\lvert \Delta \rvert \geq 14$. Thus, in either case, we must have that $\lvert \Delta \rvert \leq 28$.

Next suppose that $\Delta_3 = 4 \Delta_1 = 4 \Delta_2$. Write $\Delta = \Delta_1 = \Delta_2$. By Lemma~\ref{lem:samefield}, $\Delta \equiv 1 \bmod 8$. Thus, by Lemma~\ref{lem:dom}, there are no subdominant singular moduli of discriminant $4 \Delta$. Taking conjugates, we may assume that $x_3$ is dominant. Thus,
\[ \lvert A \rvert^{1/a} \leq \frac{(e^{\pi \lvert \Delta \rvert^{1/2}}+2079)(e^{\pi \lvert \Delta \rvert^{1/2}/2}+2079)}{e^{2 \pi \lvert \Delta \rvert^{1/2}} - 2079}.\]
There is also a conjugate $(x_1', x_2', x_3')$ with $x_3'$ not dominant and one of $x_1', x_2'$ dominant. We obtain that
\[ \lvert A \rvert^{1/a} \geq \frac{(e^{\pi \lvert \Delta \rvert^{1/2}}-2079)\min \{4.4 \times 10^{-5}, 3500 \lvert \Delta \rvert^{-3}\}}{e^{2 \pi \lvert \Delta \rvert^{1/2}/3} + 2079}.\]
The two bounds are incompatible whenever $\lvert \Delta \rvert \geq 19$.

Now suppose that $\Delta_1 = 4 \Delta_2 = \Delta_3$. Write $\Delta = \Delta_2$. By Lemma~\ref{lem:samefield}, $\Delta \equiv 1 \bmod 8$ and so, by Lemma~\ref{lem:dom}, there are no subdominant singular moduli of discriminant $4 \Delta$. Assuming that $x_1$ is dominant (and so $x_3$ is neither dominant nor subdominant), we have that
\[\lvert A \rvert^{1/a} \geq \frac{(e^{2 \pi \lvert \Delta \rvert^{1/2}}-2079)\min \{4.4 \times 10^{-5}, 3500 \lvert \Delta \rvert^{-3}\}}{e^{2 \pi \lvert \Delta \rvert^{1/2}/3} + 2079}.\]
Let $(x_1', x_2', x_3')$ be the conjugate such that $x_3'$ is dominant (and so $x_1'$ is neither dominant nor subdominant). Then we have that
\[ \lvert A \rvert^{1/a} \leq \frac{(e^{2 \pi \lvert \Delta \rvert^{1/2}/3}+2079)(e^{\pi \lvert \Delta \rvert^{1/2}}+2079)}{e^{2 \pi \lvert \Delta \rvert^{1/2}} - 2079}. \]
These bounds are incompatible whenever $\lvert \Delta \rvert \geq 8$.

Finally, suppose that $\Delta_1 = 4 \Delta_2 = 4 \Delta_3$.  Write $\Delta = \Delta_2 = \Delta_3$. By Lemma~\ref{lem:samefield}, $\Delta \equiv 1 \bmod 8$. Since $x_2 \notin \Q$, we have that $\Delta \neq -7$. Suppose that $\Delta = -15$. Then $h(4 \Delta) = h(\Delta) = 2$. The unique non-trivial Galois conjugate of $(x_1, x_2, x_3)$ is $(x_1', x_3, x_2)$, where $x_1'$ is the unique non-trivial conjugate of $x_1$. We thus have that $x_1^a = (x_1')^a$, which contradicts Proposition~\ref{prop:dom}. So $\Delta \neq -15$. So, by Lemma~\ref{lem:dom}, there are no subdominant singular moduli of discriminant $4 \Delta$, and there are precisely two subdominant singular moduli of discriminant $\Delta$.  

We may assume that $(x_1, x_2, x_3)$ has $x_3$ dominant (and so $x_2$ not dominant). Suppose for now that $x_1$ is not dominant. Then we have the upper bound
\[ \lvert A \rvert^{1/a} \leq \frac{(e^{2 \pi \lvert \Delta \rvert^{1/2}/3}+2079)(e^{\pi \lvert \Delta \rvert^{1/2}/2}+2079)}{e^{\pi \lvert \Delta \rvert^{1/2}} - 2079}.\]
And taking the conjugate $(x_1', x_2', x_3')$ with $x_1'$ dominant (and hence $x_3'$ not dominant since $(x_1', x_2', x_3') \neq (x_1, x_2, x_3)$), we have that
\[ \lvert A \rvert^{1/a} \geq \frac{(e^{2 \pi \lvert \Delta \rvert^{1/2}}-2079)\min \{4.4 \times 10^{-5}, 3500 \lvert \Delta \rvert^{-3}\}}{e^{ \pi \lvert \Delta \rvert^{1/2}/2} + 2079}.\]
These are incompatible when $\lvert \Delta \rvert \geq 13$. 

Now suppose that $x_1$ is dominant. Then $(x_1, x_2, x_3)$ has both $x_1, x_3$ dominant and gives the lower bound
\[ \lvert A \rvert^{1/a} \geq \frac{(e^{2 \pi \lvert \Delta \rvert^{1/2}}-2079)\min \{4.4 \times 10^{-5}, 3500 \lvert \Delta \rvert^{-3}\}}{e^{ \pi \lvert \Delta \rvert^{1/2}} + 2079}.\]
In this case, there exists a conjugate $(x_1'', x_2'', x_3'')$ with $x_3''$ subdominant, $x_1''$ not dominant, and $x_2''$ not dominant. Such a conjugate exists since there are two subdominant singular moduli of discriminant $\Delta$, so there are two conjugates of $(x_1, x_2, x_3)$ which have their third coordinate subdominant, and neither of these has its first conjugate dominant and at most one of these has its second coordinate dominant. This conjugate $(x_1'', x_2'', x_3'')$ gives rise to the bound
\[\lvert A \rvert^{1/a} \leq  \frac{(e^{2 \pi \lvert \Delta \rvert^{1/2}/3}+2079)(e^{\pi \lvert \Delta \rvert^{1/2}/2}+2079)}{e^{\pi \lvert \Delta \rvert^{1/2}/2} - 2079}.\]
These two bounds are incompatible when $\lvert \Delta \rvert \geq 92$. Thus, in either case, we obtain that $\lvert \Delta \rvert \leq 91$.

\subsection{The case where $\Q(x_1) = \Q(x_2) \supsetneq \Q(x_3)$.}\label{subsecB}

By the discussion at the start of the proof, we must have that $\Delta_1 = \Delta_2$ in this case.

\subsubsection{The subcase where $\Q(\sqrt{\Delta_1}) \neq \Q(\sqrt{\Delta_3})$.}

Suppose that $\Q(\sqrt{\Delta_1}) \neq \Q(\sqrt{\Delta_3})$. Then, by Lemma~\ref{lem:subfielddiffund}, either at least one of $\Delta_1, \Delta_3$ is listed in \cite[Table~1]{AllombertBiluMadariaga15} or some $h_i \geq 128$ and the corresponding field $\Q(x_i)$ is Galois. In the first case, we may find all the possible $(\Delta_1, \Delta_2, \Delta_3)$ using a PARI script. So suppose that neither of $\Delta_1, \Delta_3$ is listed. 

If $\Q(x_1)$ is Galois, then $\Q(x_3)$ is also Galois, and hence at least one of $\Delta_1, \Delta_3$ is listed in \cite[Table~1]{AllombertBiluMadariaga15} by \cite[Corollaries~2.2 \& 3.3]{AllombertBiluMadariaga15} since $\Q(\sqrt{\Delta_1}) \neq \Q(\sqrt{\Delta_3})$. We are thus actually in the first case again. So we may assume that $\Q(x_1)$ is not Galois, and hence $\Q(x_3)$ is Galois and $h(\Delta_3) \geq 128$. By Lemma~\ref{lem:tat}, we thus have that $\Q(\sqrt{\Delta_3})= K_*$, where $K_*$ is the exceptional field defined in Definition~\ref{def:tat}. We are thus in the exceptional case (5) of Theorem~\ref{thm:quot}.

\subsubsection{The subcase where $\Q(\sqrt{\Delta_1}) = \Q(\sqrt{\Delta_3})$.}

Suppose that $\Q(\sqrt{\Delta_1}) = \Q(\sqrt{\Delta_3})$. Then, by Lemma~\ref{lem:subfieldsamefund}, we have that $\Delta_1 \in \{9 \Delta_3 / 4, 4 \Delta_3, 9 \Delta_3, 16 \Delta_3\}$. Write $\Delta = \Delta_3$.

We will consider each of these four cases in turn. Our approach, which we now explain, will be the same for all of them. We may assume that $x_1$ is dominant, in order to obtain a lower bound for $\lvert A \rvert^{1/a}$. We will then show that if $h_3$ (equivalently, $h_1$) is sufficiently large, then there exists a conjugate $(x_1', x_2', x_3')$ of $(x_1, x_2, x_3)$ which gives rise to an upper bound for $\lvert A \rvert^{1/a}$ that is incompatible with the obtained lower bound for large enough values of $\lvert \Delta \rvert$. Thus, either $h_1$ is bounded, or $\lvert \Delta \rvert$ is bounded. In particular, there are only finitely many possibilities for $\Delta$, and we may find these.

In Subsection~\ref{subsecC}, we will eliminate each of these finitely many possibilities for $\Delta$ using a PARI script. The time taken to eliminate a possible value of $\Delta$ increases rapidly with $h_3 =h(\Delta)$. It is therefore computationally efficient to reduce the number of possible $\Delta$ with large class number as much as possible. 

For this reason, in each case we will obtain not one, but several, upper bounds for $\lvert A \rvert^{1/a}$, each of which is valid for a different range of $h_3$. Crucially, the implied bound on $\lvert \Delta \rvert$ becomes much sharper as $h_3$ increases. This allows us to reduce considerably the number of possible $\Delta$ with large class number, and so significantly reduces the computational task which faces us in Subsection~\ref{subsecC}. In particular, the final bound on $\lvert \Delta \rvert$ which we obtain in each case will be sufficiently sharp to rule out there being any $\Delta$ with a class number so large as to be computationally infeasible to handle.  Now we begin with the four cases themselves.

First, suppose that $\Delta_1 = \Delta_2 = 9 \Delta / 4$. Assuming that $x_1$ is dominant, we have that
\[\lvert A \rvert^{1/a} \geq \frac{(e^{3 \pi \lvert \Delta \rvert^{1/2}/2}-2079)\min \{4.4 \times 10^{-5}, 3500 \times (\frac{9}{4})^{-3}\lvert \Delta \rvert^{-3}\}}{e^{ \pi \lvert \Delta \rvert^{1/2}} + 2079}.\]
Let $k, m_1, m_2$ be as given in Table~\ref{tbl:9/4}. If $h_1 \geq k$, then we may (Lemma~\ref{lem:dom}) find a conjugate $(x_1', x_2', x_3')$ where the associated triples $(a_i', b_i', c_i') \in T_{\Delta_i}$ satisfy $a_1' \geq m_1$ and $a_2' \geq m_2$. This conjugate gives rise to the upper bound 
\[\lvert A \rvert^{1/a} \leq  \frac{(e^{3 \pi \lvert \Delta \rvert^{1/2}/2 m_1}+2079)(e^{3 \pi \lvert \Delta \rvert^{1/2}/2 m_2}+2079)}{\min \{4.4 \times 10^{-5}, 3500 \lvert \Delta \rvert^{-3}\}}.\]
These upper bounds are incompatible with the above lower bound for sufficiently large values of $\lvert \Delta \rvert$ (which are also recorded in Table~\ref{tbl:9/4}). Observe that the greater $h_1$ is, the sharper the upper bound on $\lvert \Delta \rvert$ we obtain is. We thus obtain that either $h_1 \leq 29$ (and so $h_3 \leq 14$) or $\lvert \Delta \rvert$ is bounded as in Table~\ref{tbl:9/4}. 

\begin{table}
	\caption{Conjugates when $\Delta_1 = \Delta_2 = 9 \Delta /4$.}
	\begin{tabular}{  c | c | c | c }
		
		$k$ & $m_1$ & $m_2$ & $\lvert \Delta \rvert \leq$ \\ \hline
		$29$ & $7$ & $8$ & $22443$\\
		$31$ & $8$ & $8$ & $11576$\\
		$35$ & $8$ & $9$ & $7484$\\
		$39$ & $9$ & $9$ & $5076$\\
		$42$ & $9$ & $10$ & $3820$\\
		$45$ & $10$ & $10$ & $2929$\\
		$49$ & $10$ & $11$ & $2384$\\
		$53$ & $11$ & $11$ & $1957$\\
		$55$ & $11$ & $12$ & $1669$\\
		$57$ & $12$ & $12$ & $1430$\\
	\end{tabular}
	\label{tbl:9/4}
\end{table}

Second, suppose that $\Delta_1 = \Delta_2 = 4 \Delta$. As in \cite[\S3.2.2]{BiluLucaMadariaga16}, we note that, for $m \in \Z_{>0}$, the class number formula \cite[Corollary~7.28]{Cox89} implies that
\[h(m^2 \Delta) = m \prod_{p \mid m} (1 - \frac{1}{p}\Big(\frac{\Delta}{p}\Big) ) h (\Delta),\]
where $(\Delta / \cdot)$ is the Kronecker symbol. Since $h_1 = 2 h_2$, we thus obtain that $(\Delta/2)=0$. In particular, $4 \Delta \equiv 0 \bmod 16$, and so there are no subdominant singular moduli of discriminant $4 \Delta$ by Lemma~\ref{lem:dom}. Assuming $x_1$ is dominant, we have that 
\[\lvert A \rvert^{1/a} \geq \frac{(e^{2 \pi \lvert \Delta \rvert^{1/2}}-2079)\min \{4.4 \times 10^{-5}, 3500 \times 4^{-3}\lvert \Delta \rvert^{-3}\}}{e^{ \pi \lvert \Delta \rvert^{1/2}} + 2079}.\]
Now let $k, m_1, m_2$ be as given in Table~\ref{tbl:4}. As before, provided that $h_1 \geq k$, we may find a conjugate $(x_1', x_2', x_3')$, where $x_i'$ has associated $(a_i', b_i', c_i') \in T_{\Delta_i}$, such that $a_1' \geq m_1$ and $a_2' \geq m_2$. Such a conjugate gives rise to the upper bound 
\[\lvert A \rvert^{1/a} \leq  \frac{(e^{2 \pi \lvert \Delta \rvert^{1/2}/m_1 }+2079)(e^{2 \pi \lvert \Delta \rvert^{1/2}/ m_2}+2079)}{\min \{4.4 \times 10^{-5}, 3500 \lvert \Delta \rvert^{-3}\}}.\]
We thus obtain that either $h_1 \leq 10$ (and so $h_3 \leq 5$) or $\lvert \Delta \rvert$ is bounded as in Table~\ref{tbl:4}.

\begin{table}
	\caption{Conjugates when $\Delta_1 = \Delta_2 = 4 \Delta$.}
	\begin{tabular}{  c | c | c | c }
		
		$k$ & $m_1$ & $m_2$ & $\lvert \Delta \rvert \leq$ \\ \hline
		$11$ & $5$ & $5$ & $3397$\\
		$13$ & $5$ & $6$ & $1393$\\
		$15$ & $6$ & $6$ & $650$\\
		$19$ & $6$ & $7$ & $403$\\
		$23$ & $7$ & $7$ & $293$\\
		$25$ & $7$ & $8$ & $236$\\
		$27$ & $8$ & $8$ & $194$\\
		
	\end{tabular}
	\label{tbl:4}
\end{table}

Third, suppose that $\Delta_1 = \Delta_2 = 9 \Delta$. Assuming that $x_1$ is dominant, we have that 
\[\lvert A \rvert^{1/a} \geq \frac{(e^{3 \pi \lvert \Delta \rvert^{1/2}}-2079)\min \{4.4 \times 10^{-5}, 3500 \times 9^{-3}\lvert \Delta \rvert^{-3}\}}{e^{ \pi \lvert \Delta \rvert^{1/2}} + 2079}.\]
Let $k, m_1, m_2$ be as given in Table~\ref{tbl:9}. Provided $h_1 \geq k$, we may as usual find a conjugate $(x_1', x_2', x_3')$, where $x_i'$ has associated $(a_i', b_i', c_i') \in T_{\Delta_i}$ with $a_1' \geq m_1$ and $a_2' \geq m_2$. This conjugate gives rise to the upper bound 
\[\lvert A \rvert^{1/a} \leq  \frac{(e^{3 \pi \lvert \Delta \rvert^{1/2}/m_1 }+2079)(e^{3 \pi \lvert \Delta \rvert^{1/2}/ m_2}+2079)}{\min \{4.4 \times 10^{-5}, 3500 \lvert \Delta \rvert^{-3}\}}.\]
We thus obtain that either $h_1 \leq 8$ (and so $h_3 \leq 4$) or $\lvert \Delta \rvert$ is bounded as in Table~\ref{tbl:9}.

\begin{table}
	\caption{Conjugates when $\Delta_1 = \Delta_2 = 9 \Delta $.}
	\begin{tabular}{  c | c | c | c }
		
		$k$ & $m_1$ & $m_2$ & $\lvert \Delta \rvert \leq$ \\ \hline
		$9$ & $3$ & $4$ & $2131$\\
		$11$ & $4$ & $4$ & $255$\\
		$13$ & $4$ & $5$ & $126$\\
		$15$ & $5$ & $5$ & $71$\\			
	\end{tabular}
	\label{tbl:9}
\end{table}

Finally, suppose that $\Delta_1 = \Delta_2 = 16 \Delta$. In this case, the class number formula implies that
\[h(16 \Delta) = 4 (1- \frac{1}{2}\Big(\frac{\Delta}{2}\Big))h(\Delta).\]
Since $h_1 = 2 h_3$, we obtain that $(\Delta/2)=1$ and so $\Delta \equiv 1 \bmod 8$. In particular, $16 \Delta \equiv 0 \bmod 16$, and hence there are no subdominant singular moduli of discriminant $16 \Delta$ by Lemma~\ref{lem:dom}.
Assuming that $x_1$ is dominant, we have that
\[\lvert A \rvert^{1/a} \geq \frac{(e^{4 \pi \lvert \Delta \rvert^{1/2}}-2079)\min \{4.4 \times 10^{-5}, 3500 \times 16^{-3}\lvert \Delta \rvert^{-3}\}}{e^{ \pi \lvert \Delta \rvert^{1/2}} + 2079}.\]
Let $k, m_1, m_2$ be as given in Table~\ref{tbl:16}. Provided $h_1 \geq k$, we may also find a conjugate $(x_1', x_2', x_3')$, where $x_i'$ has associated $(a_i', b_i', c_i') \in T_{\Delta_i}$ with $a_1' \geq m_1$ and $a_2' \geq m_2$. Such a conjugate gives rise to the upper bound 
\[\lvert A \rvert^{1/a} \leq  \frac{(e^{4 \pi \lvert \Delta \rvert^{1/2}/m_1 }+2079)(e^{4 \pi \lvert \Delta \rvert^{1/2}/ m_2}+2079)}{\min \{4.4 \times 10^{-5}, 3500 \lvert \Delta \rvert^{-3}\}}.\]
We thus obtain that either $h_1 \leq 4$ (and so $h_3 = 2$) or $\lvert \Delta \rvert$ is bounded as in Table~\ref{tbl:16}.

\begin{table}
	\caption{Conjugates when $\Delta_1 = \Delta_2 = 16 \Delta $.}
	\begin{tabular}{  c | c | c | c }
		
		$k$ & $m_1$ & $m_2$ & $\lvert \Delta \rvert \leq$ \\ \hline
		$5$ & $3$ & $4$ & $143$\\
		$7$ & $4$ & $4$ & $48$\\
		$9$ & $4$ & $5$ & $28$		
	\end{tabular}
	\label{tbl:16}
\end{table}

\subsection{Eliminating the exceptional cases}\label{subsecC}

If we are not in one of cases (1)--(5) of Theorem~\ref{thm:quot}, then our arguments in the previous two subsections have reduced the possibilities for $(\Delta_1, \Delta_2, \Delta_3)$ to a known finite list. We may eliminate all the entries on this list using a PARI script. The obvious modification of the algorithm described in Subsection~\ref{subsec:elimnontriv} may be used to do this. 

A slight technical difficulty arises in the case where $\Q(x_1) = \Q(x_2) \supsetneq \Q(x_3)$ and $\Delta_1 = \Delta_2 = 9 \Delta_3 /4$. Here PARI finds 30 possibilities for $(\Delta_1, \Delta_2, \Delta_3)$, of which $19$ have $h_1 \geq 30$. The run time of our usual algorithm increases with $h_1$, and it is therefore not practical to use this algorithm to check triples $(\Delta_1, \Delta_2, \Delta_3)$ with $h_1$ greater than around 30 (at least with modest computational resources).  

In this case, we therefore first filter these 30 triples with the following process. Fix such a triple $(\Delta_1, \Delta_2, \Delta_3)$. (Note that $\Delta_1 = \Delta_2$.) The singular moduli of discriminant $\Delta_1$ may be enumerated
\[ \{y_1, \ldots, y_{h_3}, y_{h_3 + 1}, \ldots, y_{2 h_3}\},\]
where $\lvert y_i \rvert \geq \lvert y_{i+1} \rvert$ for all $i$ (in particular, $y_1$ is dominant). We  may also enumerate the singular moduli of discriminant $\Delta_3$ as
\[\{w_1, \ldots, w_{h_3}\},\]
where $\lvert w_i \rvert \geq \lvert w_{i+1} \rvert$ for all $i$. This is straightforward in PARI.

Suppose then that $x_1, x_2, x_3$ is a pairwise distinct choice of singular moduli of respective discriminants $\Delta_1, \Delta_2, \Delta_3$. It must always be possible to find a conjugate $(\sigma(x_1), \sigma(x_2))$ of $(x_1, x_2)$ such that $(\sigma(x_1), \sigma(x_2)) = (y_i, y_j)$ with $i, j \geq h_3$ and $i \neq j$. We thus must have that
\[ \lvert A \rvert^{1/a} \leq \frac{\lvert y_{h_3} \rvert \lvert y_{h_3 + 1} \rvert}{\lvert w_{h_3} \rvert}.\]
We may also find a conjugate $(\sigma_0(x_1), \sigma_0(x_2))$ with $\sigma_0(x_1) = y_1$. Hence,
\[ \lvert A \rvert^{1/a} \geq \frac{\lvert y_{1} \rvert \lvert y_{2 h_3} \rvert}{\lvert w_{1} \rvert}.\]
We thus obtain upper and lower bounds for $\lvert A \rvert^{1/a}$. Note that we may obtain the same bounds for any choice of $x_1, x_2, x_3$ as pairwise distinct singular moduli of respective discriminants $\Delta_1, \Delta_2, \Delta_3$. Thus, if these upper and lower bounds for $\lvert A \rvert^{1/a}$ are incompatible, then we may reject the triple $(\Delta_1, \Delta_2, \Delta_3)$.

Implementing this approach, we are able to reject all but one of the 30 triples $(\Delta_1, \Delta_2, \Delta_3)$. The single triple $(\Delta_1, \Delta_2, \Delta_3)$ we cannot reject using this algorithm has $h_1 = 6$. This triple though has $h_1$ sufficiently small to be eliminated by the usual algorithm. The proof of Theorem~\ref{thm:quot} is thus completed in this way.

\bibliographystyle{amsplain}

\end{document}